\documentclass[11pt]{article}

\usepackage[T1]{fontenc}
\usepackage{amssymb}
\usepackage{amsmath}
\usepackage[title, toc]{appendix}
\usepackage{amsthm}
\usepackage{tikz}
\usepackage{tikz-cd}  
\usepackage{enumerate} 
\usepackage{enumitem}  
\usepackage{pdfpages}  
\usepackage{hyperref}
\usepackage{wasysym}  
\usepackage{mathtools} 
\usepackage{stmaryrd} 
\usepackage{graphicx}
\usepackage{xcolor}
\usepackage{mathrsfs}
\usepackage{titlesec}
\usepackage{todonotes} 
\usepackage{xcolor}
\usepackage[style=alphabetic, backref=true, backend=biber, hyperref=true, giveninits=true, doi=false, url=false]{biblatex}
\hypersetup{
    colorlinks,
    linkcolor={blue!50!black},
    citecolor={blue!50!black},
    urlcolor={blue!80!black}
}
\DeclareSymbolFont{extraitalic} {U}{zavm}{m}{it}
\DeclareMathSymbol{\stigma}{\mathord}{extraitalic}{168}

\newtheorem*{proposition*}{Proposition}
\newtheorem*{theorem*}{Theorem}
\newtheorem{theorem}{Theorem}[subsection]
\newtheorem{proposition}[theorem]{Proposition}

\newtheorem{lemma}[theorem]{Lemma}
\newtheorem{corollary}[theorem]{Corollary}
\newtheorem{conjecture}[theorem]{Conjecture}
\theoremstyle{remark}\newtheorem{remark}[theorem]{Remark}
\newtheoremstyle{boldremark}
    {\dimexpr\topsep/2\relax} 
    {\dimexpr\topsep/2\relax} 
    {}          
    {}          
    {\bfseries} 
    {.}         
    {.5em}      
    {}          
\theoremstyle{boldremark}
\newtheorem*{conjecture*}{Conjecture}
\makeatletter
\newcommand*\leftdash{\rotatebox[origin=c]{-45}{$\dabar@\dabar@\dabar@$}}
\newcommand*\rightdash{\rotatebox[origin=c]{45}{$\dabar@\dabar@\dabar@$}}
\makeatother

\newcommand{\bfone}{\mathbf{1}}

\newcommand{\Qlbar}{\overline{\mathbf{Q}}_{\ell}}

\newcommand{\kk}{\mathbf{k}}
\newcommand{\KK}{\mathbf{K}}
\newcommand{\Zl}{\mathbf{Z}_\ell}

\newcommand{\R}{\mathfrak{R}}

\newcommand{\GL}{\textrm{GL}}
\newcommand{\bfGL}{\mathbf{GL}}
\newcommand{\bfe}{\mathbf{e}}

\newcommand{\SL}{\textrm{SL}}
\newcommand{\bfG}{\mathbf{G}}
\newcommand{\bfB}{\mathbf{B}}
\newcommand{\bfT}{\mathbf{T}}
\newcommand{\bfU}{\mathbf{U}}

\newcommand{\n}{\mathfrak{n}}
\newcommand{\g}{\mathfrak{g}}

\newcommand{\fl}{\mathcal{B}}

\newcommand{\yy}{\mathcal{Y}}

\newcommand{\gspr}{\tilde{\g}}
\newcommand{\gstb}{\mathbf{St}_\g}

\newcommand{\cC}{\mathcal{C}}
\newcommand{\ct}{\mathcal{C}(T)(\Qlbar)}

\newcommand{\F}{\mathcal{F}}
\newcommand{\G}{\mathcal{G}}
\newcommand{\cO}{\mathcal{O}}
\newcommand{\hc}{\mathfrak{hc}}
\newcommand{\cc}{\underline{\Qlbar}}

\newcommand{\DD}{\hat{\Delta}}
\newcommand{\dd}{\hat{\delta}}
\newcommand{\NN}{\hat{\nabla}}

\newcommand{\T}{\hat{\mathcal{T}}}
\newcommand{\IC}{\operatorname{IC}}
\newcommand{\E}{\mathcal{E}}

\newcommand{\Spgr}{\mathcal{S}}
\newcommand{\Rep}{\operatorname{Rep}}
\newcommand{\RepG}{\operatorname{Rep}(\mathbf{G})}
\newcommand{\RepT}{\operatorname{Rep}(\mathbf{T})}
\newcommand{\Ho}{\operatorname{Ho}}
\newcommand{\pro}{\operatorname{pro}}
\newcommand{\M}{\mathscr{M}}

\newcommand{\Tilt}{\operatorname{Tilt}}

\newcommand{\perf}{\mathrm{perf}}
\newcommand{\coh}{\mathrm{coh}}

\newcommand{\ind}{\operatorname{Ind}}
\newcommand{\res}{\operatorname{Res}}

\newcommand{\Hom}{\operatorname{Hom}}
\newcommand{\Spec}{\operatorname{Spec}}

\newcommand{\invlim}[1]{\lim\limits_{\xleftarrow[#1]{}}}

\newcommand{\pH}{{^p\mathcal{H}}}

\newcommand{\Ad}{\operatorname{Ad}}
\newcommand{\id}{\operatorname{id}}
\titleformat{\section}
{\normalfont\bfseries}{\thesection}{1em}{}
\titleformat{\subsection}[runin]
{\normalfont\bfseries}{\thesubsection}{1em}{}

\titleformat{\subsubsection}[runin]
{\normalfont\bfseries}{\thesubsubsection}{1em}{}

 \newcommand{\Addresses}{{
  \bigskip
  \footnotesize

  K.~Tolmachov, \textsc{Department of Mathematics, University of Hamburg,
    Hamburg, Germany}\par\nopagebreak
  \textit{E-mail address}: \texttt{tolmak@khtos.com}
 }}
\DeclareSymbolFont{extraitalic} {U}{zavm}{m}{it}
\DeclareMathSymbol{\stigma}{\mathord}{extraitalic}{168}

\begin{document} \title{Linear structure on a finite Hecke category\\ in type A}

\author{Kostiantyn Tolmachov}
\date{}
\maketitle
\abstract{For the group $\GL_n$, we construct an action of the equivariant derived category of coherent sheaves on the Grothendieck-Springer resolution on a certain subcategory of a finite monodromic Hecke category. We use this to construct a partial categorification of the projection from the extened affine to the finite Hecke algebra of $\GL_n$. As a crucial intermediate step, we compute the exterior powers, with respect to the perversely truncated multiplicative convolution, of a parabolic Springer sheaf corresponding to a maximal parabolic subgroup fixing a line in the defining $n$-dimensional representation of $\GL_n$.}
\tableofcontents
\section{Notations.}
Fix primes $\ell\neq p$. Let $\kk$ be an algebraically closed field of characteristic $p$, and let $\KK = \Qlbar$. All triangulated categories considered are assumed to be linear over $\KK$.

We will work with stacks defined over $\kk$ and $\KK$ appearing on different sides of Langlands duality. We will use normal typeface $X, G, \dots $ for geometric objects over $\kk$, and bold typeface $\mathbf{X}, \mathbf{G}$ for geometric objects over $\KK$.

For a stack $X$ defined over $\kk$, let $D^b(X)$ be the bounded $\ell$-adic derived category of $\Qlbar$-sheaves with constructible cohomology. Let $P(X)$ be the abelian category of perverse sheaves on $X$ with respect to the middle perversity. For a group $H$ acting on a scheme $X$, let $D^b_H(X) = D^b(X/H)$ and $P_H(X)$ be the corresponding equivariant categories, defined as in \cite{bernsteinEquivariantSheavesFunctors2006}, \cite{laszloSixOperationsSheaves2008a}. Let $\cc_{X}$ stand for the constant sheaf on $X$.

For an arbitrary stack $\mathbf{X}$ over $\KK$, write $\coh(\mathbf{X})$ and $D^b_{\coh}(\mathbf{X})$ for the category of coherent sheaves and bounded coherent derived category of sheaves on $\mathbf{X}$, respectively, and let $\perf(\mathbf{X})\subset D^b_{\coh}(\mathbf{X})$ stand for the perfect subcategory. For a group $\mathbf{H}$ acting on a scheme $\mathbf{X}$, let $D^b_{\coh,\mathbf{H}}(\mathbf{X})=D^b_\coh(\mathbf{X}/\mathbf{H})$ and  $\perf_{\mathbf{H}}(\mathbf{X})$ be the corresponding equivariant categories.

For reductive groups $G, \mathbf{G}$, defined over $\kk$ or over $\KK$, respectively, we will write $B, \bfB$ for the chosen Borel subgroups, $T \subset B, \bfT \subset \bfB$ for the maximal tori and $U \subset B, \bfU \subset \bfB$ for the unipotent radicals of $B, \bfB$. Let $\g$ stand for the Lie algebra of $\bfG$.

We write $\RepG$ for the symmetric monoidal category of finite dimensional representations of $\bfG$ defined over $\KK$, so that $\RepG \simeq \coh(\Spec(\KK)/\bfG)$.

For $\bfG = \bfGL_n$ we identify the weight lattice $\Lambda$ of $\bfG$ with $\mathbb{Z}^n$ with basis $\bfe_1, \dots, \bfe_n$, so that the Weyl group $W = S_n$ of $\bfGL_n$ acts by permutations of this basis, and simple roots are given by $\alpha_i = \bfe_{i+1}-\bfe_{i}$ for $i = 1, \dots, n-1$. Write $s_i = s_{\alpha_i} \in W$. We adopt similar conventions for $G = \GL_n$, defining a basis $e_1,\dots, e_n$. Let $w_0$ stand for the longest element in $W$.

We always assume that the characteristic $p$ of $\kk$ is good for $G$.
\section{Introduction.}
\subsection{Linear structures.} Let $\mathcal{T}$ be a $\KK$-linear triangulated category, and let $X$ be a stack defined over $\KK$. We call an exact action of the monoidal category $\perf(\mathbf{X})$ on $\mathcal{T}$ a linear structure on $\mathcal{T}$ over $X$ .

 Let $G, \bfG$ be reductive groups defined over $\kk$ and $\KK$, respectively.

Let $\mathcal{H}(G)$ stand for the completed unipotently monodromic Hecke category of a reductive group $G$, defined  in \cite{bezrukavnikovKoszulDualityKacMoody2013}. This is a certain completion of a full subcategory of unipotently monodromic complexes in $D^b(U\backslash G/U)$. The category $\mathcal{H}(G)$ is monoidal, with the monoidal structure being the convolution operation, which we denote by $- \star -$.

Let $\g$ be the Lie algebra of $\bfG$, let $\mathfrak{b}$ be the Lie algebra of a Borel subgroup $\bfB \subset \bfG$, and let $\gspr =  \bfG\times^{B} \mathfrak{b}$ be the Grothendieck-Springer resolution. 
The goal of this paper is to prove that the Hecke category $\mathcal{H}_n = \mathcal{H}(\GL_n)$ associated with the general linear group $G = \GL_n$ admits a linear structure over $\gstb/\bfG$, where $\bfG = \bfGL_n,$ and $\gstb = \gspr\times_\g\gspr$ is the Steinberg variety associated with the Lie algebra $\g = \mathfrak{gl}_{n}$. For the rest of the introduction, let $G = \GL_n, \bfG = \bfGL_n$.

Let $\mathcal{H}_n^{\perf}$ be the full triangulated subcategory of $\mathcal{H}_n$ generated by the big free-monodromic tilting sheaf $\T_{w_0}$ -- see Section \ref{sec:finite-monodr-hecke}. We will equip $\mathcal{H}_n^{\perf}$ with  two compatible linear structures over the stack $\gspr/\bfG$.

The following theorem is one of our main results (Theorem \ref{sec:extens-funct-from-1} in the main text).
\begin{theorem}
 There is an exact monoidal functor 
\begin{equation}
  \label{eq:1}
 \varpi: \perf_G(\gstb) \to \mathcal{H}_n, 
\end{equation}
\sloppy where the monoidal structure on the perfect equivariant category $\perf_G(\gstb)$ is given by convolution. The functor takes values in $\mathcal{H}_n^{\perf}$ and is compatible with the linear structures over $\gspr/\bfG$.
\end{theorem}

\subsection{Tannakian formalism for the Steinberg variety.} We proceed to outline the main ideas of our construction of the linear structures and the functor $\varpi$.
The construction employs the Tannakian formalism for $D^b_{\coh}(\gspr/\bfG), \perf_\bfG(\gstb)$ developed by Bezrukavnikov in \cite{bezrukavnikovTwoGeometricRealizations2016}. Roughly, this formalism allows us to construct a linear structure over $\gspr/G$ on an arbitrary monoidal triangulated category $\mathcal{T}$ by equipping $\mathcal{T}$ with an action of $\Rep(\bfG \times \bfT)$ satisfying certain compatibilities. If we have two such structures with isomorphic actions of two copies of $\RepG$, we get a monoidal functor $\perf_{\bfG}(\gstb) \to \mathcal{T}$.
\begin{remark}
  In \cite{bezrukavnikovTwoGeometricRealizations2016}, this formalism is used to construct a monoidal equivalence between $D^b_{\coh, \bfG}(\widehat{\gstb})$, where $\widehat{\gstb}$ stands for a certain formal completion of $\gstb$, and the completed monodromic affine Hecke category $\mathcal{H}_{G^{\vee}}^{\mathrm{aff}}$. We do not know if it is possible to extend the functor $\varpi$ to $D^b_{\coh,\bfG}(\widehat{\gstb})$. 
\end{remark}
\subsection{Homomorphism between affine and finite Hecke algebras in type A.}
Construction of the functor $\varpi$ is a step towards a categorification of a homomorphism from the affine to the finite Hecke algebra of $\GL_n$. 

The homomorphism in question may be most readily seen on the level of the corresponding braid groups. If we interpret the braid group $B_n$ on $n$ strands as the fundamental group of the configuration space $\textrm{Conf}_n(\mathbb{C})$ of $n$ points in the complex plane $\mathbb{C}$ based at a configuration avoiding $0$, and the affine braid group ${B}_n^{\textrm{aff}}$ as the fundamental group of the configuration space $\textrm{Conf}_n(\mathbb{C}^*)$ of $n$ points in the punctured complex plane $\mathbb{C}^*$ based at the same configuration, then the inclusion $\mathbb{C}^* \to \mathbb{C}$ gives a homomorpism
\[
  \varpi_{\textrm{br}}:{B}_n^{\textrm{aff}}\simeq \pi_1(\textrm{Conf}_n(\mathbb{C}^*)) \to \pi_1(\textrm{Conf}_n(\mathbb{C})) \simeq B_n.
\]
This descends to a homomorphism
\[
  \varpi_{\textrm{alg}}:{H}_n^{\textrm{aff}} \to H_n
\]
between the corresponding extended affine Hecke algebra ${H}^{\textrm{aff}}_n$ of the reductive group $\bfGL_n$ and the finite Hecke algebra $H_n$. Recall that ${H}^{\textrm{aff}}_n$ contains $H_n$ as a subalgebra. It also contains a commutative subalgebra, defined by Bernstein, which is identified with the group algebra of the (co-)weight lattice of $\bfGL_n$. Let $\theta_1, \dots, \theta_n$ be the generators of this subalgebra corresponding to the basis elements $\bfe_1, \dots, \bfe_n$. The homomorphism $\varpi_{\textrm{alg}}$ is defined uniquely by the following two properties: $\varpi_{\textrm{alg}}(x) = x$ for $x \in H_n$, and $\varpi_{\textrm{alg}}(\theta_1) = 1.$ The images of the elements $\theta_k$ in the finite Hecke algebra are known as multiplicative Jucys--Murphy elements therein, and play an important role in the representation theory of $H_n$, analogous to the role of classical Jucys--Murphy elements for symmetric groups -- see \cite{isaevRepresentationsHeckeAlgebras2005}, \cite{okounkovNewApproachRepresentation1996}.  

The $\Rep(\bfT)$-action employed in the Tannakian formalism of \cite{bezrukavnikovTwoGeometricRealizations2016} described above corresponds to the action of the commutative subalgebra generated by $\theta_i^{\pm 1}$. The action of $V \in \Rep(\bfG)$ corresponds to the action of the central element in $H^{\textrm{aff}}_n$ corresponding to $V$ -- see \cite{lusztigSingularitiesCharacterFormulas1983}, \cite{gaitsgoryConstructionCentralElements2001}. This element can be expressed as a symmetric Laurent polynomial in $\theta_i$, according to the character of $V$, which is expressed on the categorical level as the compatibility condition on $\Rep(\bfT)$- and $\Rep(\bfG)$- actions mentioned above.

\begin{remark}
  To speak of a categorification of the homomorphism $\varpi_{alg}$ in the setting of geometric Hecke categories, one needs to consider ``mixed'' versions of these categories. We do not treat the ``mixed'' ca\-tegorification in the present article. We expect, however, that methods of \cite{ho2022revisiting} can be used to obtain ``mixed'' versions of our results.
\end{remark}

\subsubsection{Action of $\RepT$.}
\label{sec:action-reptk} The action of $\Rep(\bfT)$ is defined in Section \ref{sec:jucys-murphy-sheav} using the description of the homomorphism $\varpi_{\textrm{br}}$ and the braid group action on $\mathcal{H}_n$ defined by Rouquier \cite{rouquierCategorificationBraidGroups2004}. To any braid $\beta \in B_n$ one can attach an object $F_\beta \in \mathcal{H}_n$ such that the action of $\beta \in B_n$ is isomorphic to the functor of convolution with $F_\beta$. For any weight ${\boldsymbol\lambda}$ there is a lift $\tilde{\theta}({\boldsymbol\lambda}) \in B_n^{\textrm{aff}}$ of the Bernstein commutative subalgebra element $\theta({\boldsymbol\lambda}) \in H_n^{\textrm{aff}}$ attached to ${\boldsymbol\lambda}$. We define the action of ${\boldsymbol\lambda}$, considered as an object in $\Rep(\bfT)$, to be the functor corresponding to the action of the braid $\varpi_{\textrm{br}}(\tilde{\theta}({\boldsymbol\lambda}))$. It is isomorphic to the functor of convolution with $F_{\varpi_{\textrm{br}}(\tilde{\theta}({\boldsymbol\lambda}))}$. This complex can be expressed as a convolution of complexes of the form $F_\beta$ attached to the Jucys--Murphy braids.      
\subsubsection{Action of $\RepG$.}  
\label{sec:action-repg}
Recall that the category $D^b(G)$ can be equipped with a monoidal structure as follows. Let $m:G\times G \to G$ be the multiplication map. For $\F, \G \in D(G)$, write
\begin{equation}
\label{eq:8}
  \F \star \G = m_*(\F\boxtimes\G).
\end{equation}
The operation $\star$ is sometimes called $\ast$-convolution, to distinguish it from the dual operation using $m_!$ instead of $m_*$ in its definition. In this paper, we will work with the $\ast$-convolution $-\star-$ and refer to it simply as convolution. This operation is a categorification of the operation of convolution of functions on $G$.

The equivariant category $D^b_G(G)$, where $G$ acts on itself by the adjoint action, can be equipped with the compatible monoidal structure. The resulting monoidal category is canonically braided. In particular, for any pair of objects $\F, \G \in D^b_G(G)$, there is a braiding isomorphism 
\[
  \beta_{\F,\G}:\F\star\G \to \G\star\F.
\]
 Moreover, $D^b_G(G)$ is equipped with an endomorphism of the identity functor, \[\theta:\operatorname{Id}_{D^b_G(G)}\to\operatorname{Id}_{D^b_G(G)},\] called a twist. The twist $\theta$ is related to the braiding by the formula
\[
  \theta_{\F\star\G} = \beta_{\G,\F}\circ\beta_{\F,\G}\circ(\theta_{\F}\star\theta_G).
\]
See \cite{boyarchenkoCharacterSheavesUnipotent2013} for a detailed discussion.

The convolution $\star$ on $D^b(G)$ is right $t$-exact with respect to the perverse $t$-structure, since $m$ is an affine morphism. This allows us to equip $P(G), P_G(G)$ with a perversely truncated convolution
 \[
   \F \star^0 \G = \pH^0(\F\star\G).
 \]
This makes $P_G(G)$ a braided monoidal category as well.
The canonical braided structure on $D^b_G(G)$ is almost never symmetric, even in the case when $G$ is commutative.
 
It is well known that $G = \GL_n$ satisfies the following property: for any $g\in G$, the centralizer $Z_G(g)$ is connected. We check in Section \ref{sec:conv-reduct-group} that in this case $P_G(G)$ with the monoidal structure $\star^0$ is symmetric braided. This essentially follows from the fact that the twist automorphism $\theta$ acts by the identity on perverse objects. Consequently, we can speak of various Schur functors in this category, in particular of exterior powers $\wedge^k\F$ for any object $\F \in P_G(G)$. In order to construct the action of $\Rep(\bfG)$,  we compute the exterior powers of a particular object called the \emph{parabolic Springer sheaf} (with respect to a particular parabolic subgroup).  

Let $P \subset G$ be a parabolic subgroup. Let $U = U_P$ stand for the unipotent radical of $P$. We write $x \mapsto {}^gx := gxg^{-1}.$ Define
\[
  \tilde{\mathcal{N}}_{P} = \{(x, gP)\in G\times G/P: x \in {}^gU\}.
\]
Let $\pi: \tilde{\mathcal{N}}_P \to G$ be the natural projection, which is known to be a semismall map, and define a perverse sheaf $\Spgr_P = \pi_*\cc_{\tilde{\mathcal{N}}_P}[2 \dim U].$
Let $\mathcal{N}_G$ stand for the variety of unipotent elements in $G$. The sheaves $\Spgr_P$ are supported on $\mathcal{N}_G.$

Now let $V = \kk^n$. The Weyl group $W \simeq S_n$ acts on $V$ by permuting the standard basis. Let $W_1 \subset W$ be the stabilizer of the first basis vector. Let $G = \GL_n = \GL(V)$, let $B$ be the Borel subgroup stabilizing the standard complete flag in $V$, and let $P_1\supset B$ be the parabolic subgroup fixing the line in $\mathbb{P}(V)$ spanned by the first basis vector. Write
\[
  \Spgr_n := \Spgr_{P_1}.
\]
Let
\[
  \Rep(W) \to D^b_G(\mathcal{N}_G), V' \mapsto \Spgr_{V'}
\]
stand for the functor given by the Springer correspondence \cite{brylinskiTransformationsCanoniquesDualite1986}, \cite{borhoRepresentationsGroupesWeyl1981}. We follow the conventions of \cite{brylinskiTransformationsCanoniquesDualite1986}, so that the trivial representation of $W$ corresponds to the punctual perverse sheaf at the unit $e \in \mathcal{N}_G$. Since $V \simeq \operatorname{Ind}_{W_1}^W \operatorname{triv}_{W_1},$ we have $\mathcal{S}_n \simeq \mathcal{S}_V$ \cite{borhoRepresentationsGroupesWeyl1981}. Write $\Spgr_{\wedge^k}:=\Spgr_{\wedge^k V}$.  

We will look for a $\Rep(\bfG)$-action on $\mathcal{H}_n^{\perf}$ such that the convolution with $\mathcal{S}_n$ gives the action of the defining representation $\mathbf{V}$ of $\bfG = \bfGL(\mathbf{V})$. Our main computational result is the following
\begin{theorem}
  \label{sec:stat-main-results}

  Exterior powers of $\Spgr_n$ with respect to the symmetric monoidal structure $\star^0$ satisfy the following: for $k \geq 0,$ there is an isomorphism
  \[\wedge^k\Spgr_n \simeq \Spgr_{\wedge^k}.\]
\end{theorem} 

It is well-known that the category $D^b_G(G)$ acts on $\mathcal{H}_n$ by left and right convolution. Moreover, this action is central, meaning that the action on the right can be identified with the action on the left, using the equivariance with respect to adjoint action. It is thus natural to look for a $\RepG$-action on $\mathcal{H}_n$ that factors through the action of $D^b_G(G)$.

Let $\RepG^{+}$ stand for the symmetric monoidal category of \emph{polynomial} representations of $\bfG = \bfGL_n$. We have the following corollary of Theorem \ref{sec:stat-main-results} (Section \ref{sec:repgk+-repgk-action-1} and the proof of Corollary \ref{sec:repgk+-repgk-action}).
\begin{corollary}
  \label{sec:action-repgk}
There is a symmetric monoidal functor $\RepG^+ \to P_G(G)$ sending the defining representation $\mathbf{V}$ of $\bfG$ to $\mathcal{S}_n$.   
\end{corollary}
In Corollary \ref{sec:repgk+-repgk-action} it is shown that the subcategory $\mathcal{H}_n^{\perf}$ admits an action of $P_G(G)$ by convolution, and that the $\RepG^+$-action on $\mathcal{H}_n^{\perf}$ obtained can be extended to the action of $\RepG$. 
\begin{remark}
 It would be interesting to consider the values of other Schur functors applied to $\mathcal{S}_n$. The author, however, does not know how to approach their computation, even in the ``linearized'' case of sheaves on the Lie algebra $\g$, as in \cite{bezrukavnikovTolmachovExt}. In that case, for example, the author does not expect other Schur functors in the derived category to be perverse and does not see a reason to expect them to be semisimple.    
\end{remark}
\begin{remark}
  We note that the functor $\varpi$ constructed with the help of the $\Rep(\bfG)$-action above turns out to have a very simple description, see Section~\ref{sec:comp-with-mono}. The non-trivial part of the present paper is the construction of the $\Rep(\bfG)$-action via the exterior powers computation, and its various compatibilities. 
\end{remark}

\subsubsection{Jucys--Murphy filtrations and monodromy.}  To complete the construction of the linear structure and the functor $\varpi$ one needs to establish the compatibility of $\RepG$ and $\RepT$ actions. This is done in Section~\ref{sec:extens-funct-from-2}. The main idea is to use the fact that the objects $\wedge^{k}\Spgr_n$ come from the parabolic induction from the Levi subgroup $\GL_k \times \GL_{n-k}$ of $G$. This allows us to construct filtrations of the functors defining the $\RepG$-action by the functors defining the $\RepT$-action categorifying the weight filtration. The monodromy transformations required by the Tannakian formalism of \cite{bezrukavnikovTwoGeometricRealizations2016} come from the monodromy along the central $\mathbb{G}_m \subset \GL_k$.  
\subsection{Related work.} Patterns suggesting the existence of a categorification of the homomorphism $\varpi_{\textrm{alg}}:H_n^{\mathrm{aff}} \to H_\n$ play a prominent role in the study of categorified link invariants \cite{gnr}. This was one of the starting points of our investigation. The (``mixed'' analogues of) objects  $\wedge^k\mathcal{S}_n$ (more precisely, their projection to the category of unipotent character sheaves) are objects representing Khovanov--Rozansky homology in the setting of monodromic Hecke categories \cite{hhh}.

In the setting of matrix factorizations, the categorification of $\varpi_{\textrm{br}}$  was constructed by Oblomkov and Rozansky in \cite{oblomkovAFFINEBRAIDGROUP2019}.   

Most of the results and ideas of this paper are contained in the author's thesis \cite{tolmachovFunctorAffineFinite2018}. The main difference is the replacement of the monoidal derived category of character sheaves with the abelian category, and technical complications arising from this.
This replacement is needed because the derived category $D^b_{\GL_n}(\GL_n)$ is not \emph{symmetric} braided, but was mistakenly treated as such in ibid.\@ We expect, however, that the braiding on the derived category can be modified to make it symmetric and make the simpler argument of \cite{tolmachovFunctorAffineFinite2018} go through. See the discussion in Section \ref{sec:conv-reduct-group}.

In the note \cite{bezrukavnikovTolmachovExt}, joint with Bezrukavnikov, we consider a simpler analogue of some of the results presented here in the setting of the derived category of sheaves on the Lie algebra $\mathfrak{gl}_n$.

A categorification of $\varpi_{\textrm{alg}}$ whose source is an additive diagrammatic affine Hecke category is constructed in \cite{MACKAAY2024109401}. See also \cite{eliasGaitsgoryCentralSheaves2018} for a discussion of the diagrammatic approach to the triangulated categorification $\varpi_{\textrm{alg}}$.

A different approach to constructing a functor from the affine Hecke category is developed in \cite{taoAffineHeckeCategory2021}, which presents the monoidal $\infty$-version of the Hecke category as a monoidal colimit of the finite type Hecke categories associated with parahoric subgroups. The description of $\varpi_{\textrm{alg}}$ in terms of braid groups provides a natural candidate for the restriction of its categorification to these finite type Hecke categories. See \cite{tao_talk} for the discussion of the compatabilities between these restrictions required to extend them to the affine Hecke category.  

\subsection{Acknowledgments.} I would like to thank {Roman} {Bezrukavnikov} for many helpful conversations and his encouraging attention to this project.
One of the main ideas of this paper, the inductive use of Mellin transform in the proof of Theorem \ref{sec:lie-group-case}, is due to him. I would like to thank Pavel Etingof, Michael Finkelberg and Ivan Mirković for helpful discussions.

I would like to thank the anonymous referees for a careful reading of the manuscript and many suggestions on improving the presentation.

\sloppy The author was supported by the EPSRC programme grant EP/R034826/1 and by the Deutsche Forschungsgemeinschaft SFB 1624 grant, Projektnummer 506632645.
\section{Finite monodromic Hecke category.}
\label{sec:finite-monodr-hecke}
In this section, we collect some definitions and statements about the finite Hecke categories. 

\subsection{Free-monodromic objects.}
\label{sec:free-monodr-objects}
Let $G$ be an arbitrary reductive group. Let $\mathcal{H}'(G)$ be the full triangulated subcategory of $D^b(U\backslash G/U)$ generated by the image of the functor $\pi^*:D^b(U\backslash G/B) \to D^b(U\backslash G/U),$ where $\pi$ is the map induced by the standard projection $G/U \to G/B$. Complexes in $\mathcal{H}'(G)$ are characterized among all constructible sheaves on $U \backslash G/U$ by the property that their restriction to the fibers of the $T$-torsor $G/U \to G/B$ have smooth cohomology with unipotent monodromy.

In \cite{bezrukavnikovKoszulDualityKacMoody2013}, a completion $\mathcal{H}(G) \subset \pro\mathcal{H}'(G)$ was defined. This is a triangulated monoidal category, equipped with a convolution operation $\star$, coming with three collections of objects $\DD_w, \NN_w, \T_w,$ indexed by the elements $w$ in the Weyl group $W$. These objects are referred to as free-monodromic standard, costandard and tilting objects, respectively. 

We collect some facts about these objects in the following
 \begin{proposition}[\cite{bezrukavnikovKoszulDualityKacMoody2013}, \cite{zbMATH07610537}]~
 \label{braid_relations}
  \begin{enumerate}[label=\alph*),ref=(\alph*)]
  \item\label{item:11} $\dd = \DD_{e}$ is the unit of the monoidal structure $\star$.
  \item\label{item:12} $\DD_v \star \DD_w \simeq \DD_{vw}$ and $\NN_v \star \NN_w \simeq \NN_{vw}$ if $l(vw) = l(v) + l(w)$.
  \item\label{item:13} $\DD_v \star \NN_{v^{-1}} \simeq \dd.$
  \item\label{item:14}\(\DD_w\star\T_{w_0}\simeq\T_{w_0}\star\DD_w\simeq \T_{w_0}\) and \(\NN_w\star\T_{w_0}\simeq\T_{w_0}\star\NN_w\simeq \T_{w_0}.\)
  \end{enumerate}
 \end{proposition}
 \begin{proof}
   \ref{item:11} and \ref{item:12} are  \cite[Lemma 4.3.3]{bezrukavnikovKoszulDualityKacMoody2013}, \ref{item:13} is \cite[Lemma 7.7]{zbMATH07610537}, \ref{item:14} is \cite[Proposition 7.10]{zbMATH07610537}. 
 \end{proof}
Note that the equivalence of \cite{bezrukavnikovTwoGeometricRealizations2016} assigns to the structure sheaf $\cO_{\gstb}$ the free-monodromic tilting object $\T_{w_0}$ on the finite Hecke subcategory. It can be shown that the perfect derived category of $\gstb$ is generated, as an idempotent-complete triangulated category, by the two sided monoidal ideal generated by $\cO_{\gstb}$. This motivates the following definition. Let $\mathcal{H}^{\perf}(G)$ stand for the full triangulated subcategory of $\mathcal{H}(G)$ generated by the object $\T_{w_0}$. It is easy to deduce from Proposition \ref{braid_relations} that $\mathcal{H}^{\perf}(G)$ is a two-sided monoidal ideal in $\mathcal{H}(G)$.
 \begin{remark}
   By the results of \cite{bezrukavnikovKoszulDualityKacMoody2013}, we have an equivalence of categories $\mathcal{H}(G) \simeq \Ho^b(\Tilt^{\wedge}(G))$, where the category on the right-hand side is the bounded homotopy category of the full additive subcategory $\Tilt^{\wedge}(G)$ of $\mathcal{H}(G)$ consisting of free-monodromic tilting objects. Under this equivalence, $\mathcal{H}^{\perf}(G)$ corresponds to the subcategory of complexes whose terms are direct sums of copies of $\T_{w_0}$.  
 \end{remark}
 \subsection{Jucys--Murphy sheaves and the $\Rep(\bfT)$-action.}
 \label{sec:jucys-murphy-sheav}
 Let $G = \GL_n$ and write $\mathcal{H}_n = \mathcal{H}(G)$ and $\mathcal{H}_n^{\perf} = \mathcal{H}^{\perf}(G)$. Let $B(W)$ stand for the Artin-Tits braid group attached to the Coxeter group $W$. By the result of Rouquier \cite{rouquierCategorificationBraidGroups2004}, there are left and right actions of $B(W)$ on $\mathcal{H}_n$. That is, there is a monoidal functor from the monoidal category with the set of objects $B(W)$, identity morphisms only and monoidal structure given by multiplication (opposite multiplication for the right action) in $B(W)$, to the monoidal category of endofunctors of $\mathcal{H}_n$. The endofunctor $\Phi_l(\beta)$ (respectively, $\Phi_r(\beta)$) corresponding to $\beta \in B(W)$ is isomorphic to the functor of left (respectively, right) convolution with a complexes of sheaves $F_\beta \simeq \Phi_l(\beta)(\dd).$ Let $\tilde{w}$ stand for the standard lift of $w \in W$ to $B(W)$. Then we have $F_{\tilde{w}} \simeq \DD_w, F_{\tilde{w}^{-1}} \simeq \NN_{w^{-1}}$. Note that by Proposition \ref{braid_relations}, these objects satisfy braid relations under convolution.

Define the braids $j_k, k = 1, \dots, n$ by $j_1 = e, j_{k+1} = \tilde{s}_kj_k\tilde{s}_k$. The braids $j_k$ are called Jucys--Murphy braids, and it is well known that they generate a commutative subgroup of $B(W)$. We now return to the notations of the introduction, where $\bfG = \bfGL_n$ and $\bfT$ is a maximal torus of $\bfG$. For a weight $\boldsymbol\lambda$ we denote by $\boldsymbol\lambda$ the corresponding object in $\RepT$. The assignment
\[
  \boldsymbol\lambda = \sum a_i\bfe_i \mapsto \Phi_l(j_1^{a_1}\dots j_n^{a_n}),
\]
defines the left action of $\RepT$ on $\mathcal{H}_n$. It follows from the discussion above, that the action of $\boldsymbol\lambda$ is isomorphic to the left convolution with an object
\[
  \mathfrak{J}_{\boldsymbol\lambda} = F_{j_1}^{\star a_1}\star  \dots\star F_{j_n}^{\star a_n}.
\]
The right action is defined analogously. 

By Proposition \ref{braid_relations} \ref{item:14} this action preserves the subcategory $\mathcal{H}_n^{\perf}$.
\section{Construction of the $\RepG$-action.}

\subsection{Mellin transform.} We will use extensively some results of
\cite{gabberFaisceauxPerversLadiques1996}. Let $T$ be an algebraic torus of rank $r$ defined over $\kk$. Let $\pi^t_1(T)$ stand for the tame fundamental group of $T$ and $\pi_1^\ell(T)$ for its maximal pro-$\ell$ quotient. Let \[\mathfrak{R} = \Qlbar\otimes_{\Zl}\Zl[[t_1,\dots,t_r]].\] By Proposition 3.2.2 of \cite{gabberFaisceauxPerversLadiques1996} $\R$ is a regular Noetherian ring. Let \[\mathcal{C}(T)_\ell = \Spec\mathfrak{R}\] and let \[\mathcal{C}(T) = \bigsqcup_{\chi}\{\chi\} \times \mathcal{C}(T)_\ell,\] where $\chi$ runs through the characters $\chi:\pi^t_1(T) \to \Qlbar^{\times}$ of finite order prime to $\ell$.
Let $\mathbf{1}$ stand for the trivial character of $\pi_1^t(T)$. $\Qlbar$-points of $\cC(T)$, resp.\@ $\cC(T)_\ell$, are identified with characters $\pi_1^t(T) \to \Qlbar^{\times}$, resp.\@ with continuous characters $\pi_1^\ell(T) \to \Qlbar^{\times}$. The group $\mathcal{C}(T)(\Qlbar)$ acts on the scheme $\mathcal{C}(T)$ by automorphisms. Denote by $m_{\lambda}$ the automorphism corresponding to $\lambda \in \cC(T)(\Qlbar)$.

For a character $\lambda \in \cC(T)(\Qlbar)$ write $\mathcal{L}_\lambda$ for the corresponding local system on $T$ of rank 1. More generally, let $\mathfrak{m}$ be the augmentation ideal of the completed group algebra $\Zl[[\pi_1^\ell(T)]]$, let $\E_n$ be the local system corresponding to the representation $\Qlbar\otimes_{\Zl}\Zl[[\pi_1^{\ell}(T)]]/\mathfrak{m}^n$ and let
\[
  \E_n^{\lambda} = \mathcal{L}_\lambda\otimes \E_n.
\]
Note that $\mathcal{L}_\lambda = \E_1^{\lambda}.$

 Recall that in \cite{gabberFaisceauxPerversLadiques1996} the Mellin transform functor \[ \M = \M_*:D^b(T) \to D^b_{\coh}(\cC(T)),\] was defined, satisfying the following properties:
\begin{enumerate}[label=M\arabic*.,ref=M\arabic*]
\item $\M$ is conservative, i.e. $\M(\F) = 0$ for $\F \in D^b(T)$ implies $\F = 0$. \label{item:3}
\item $\M$ is $t$-exact with respect to the perverse $t$-structure on $D^b(T)$ and standard $t$-structure on $D^b_{\coh}(\cC(T))$. \label{item:4}
\item $\M$ is monoidal with respect to the $*$-convolution on $D^b(T)$ and derived tensor product on $D^b_{\coh}(\cC(T))$, i.e.
  \[
    \M(\F\star\mathcal{G}) =
\M(\F)\otimes^{\operatorname{L}}\M(\mathcal{G}).\] \label{item:5}
\item $\M$ satisfies
  \[\M(\mathcal{L}_{\lambda}\otimes\F) =
    m_{\lambda}^*\M(F).
  \] \label{item:6}
\item $\M$ restricts to an equivalence
  \[
    \M: D_{mon}^b(T) \to D^b_{\coh,f}(\cC(T)),
  \]
where $D_{mon}^b(T)$ stands for the subcategory of monodromic complexes on $T$ and $D^b_{\coh,f}(\cC(T))$ stands for the subcategory of complexes with finite support. \label{item:7}
\end{enumerate}

Let $\mathfrak{m}_0$ be the ideal of $\R$ generated by $t_1,\dots,t_r$. We regard $\R/\mathfrak{m}_0^n$ as a coherent sheaf on $\cC(T)$ set-theoretically supported on a point inside the connected component $\{\mathbf{1}\}\times\cC(T)_\ell.$ We record the following simple
\begin{proposition}
  We have an isomorphism
  \[
    \M(\E_n^{\lambda}[r]) = m_{\lambda}^*(\R/\mathfrak{m}_0^n).
  \]
\end{proposition}
Let $\mathbf{N}$ be the category whose objects are non-negative integers, with $\Hom(i,j)$ having a single arrow if and only if $i \geq j$. Let $\mathcal{A}$ be an additive category and let $\F:\mathbf{N}\to\mathcal{A}$ be an inverse system in $\mathcal{A}$. Such a system is called essentially zero (short for ``essentially constant equal to zero'') if, for any $m \in \mathbb{N}$, there is $n > m$ such that $\F(n \rightarrow m) = 0$.

We will need the following
\begin{proposition}
 \label{sec:mellin-transform} For $\F \in D^b(T)$ the following are equivalent:
  \begin{enumerate}[label=\alph*),ref=(\alph*)]
  \item\label{item:20} $\F \in {{}^pD^{\geq 0}}$,
  \item\label{item:21} The system of $^p\mathcal{H}^i(\F\star\E_n^{\lambda}[r])$ is essentially zero for all $\lambda \in \mathcal{C}(T)(\Qlbar)$ and all $i < 0$.
  \end{enumerate}
\end{proposition}
\begin{proof}
  \ref{item:20} $\implies$ \ref{item:21}. Since the functor $(-)\star\E_n^{\lambda}[r]$ is $t$-exact from the right, it is enough to prove the statement for $\F \in P(T)$. We may also assume $\lambda = \mathbf{1}$. By properties \ref{item:4} and \ref{item:5} we have
  \[
    \M(^p\mathcal{H}^i(\F\star\mathcal{E}_n[r])) = \mathcal{H}^i(\M(\F)\otimes^{\operatorname{L}}(\R/\mathfrak{m}_0^n)).
  \]
  Note that since $\R$ is regular Noetherian and the automorphism group $\cC(T)_\ell(\Qlbar)$ acts transitively on the set of $\Qlbar$-points, $\cC(T)_\ell$ has finite global dimension (see \cite[\href{https://stacks.math.columbia.edu/tag/00OE}{Tag 00OE}]{stacks-project}), and $\M(\F)|_{\mathbf{1}\times\cC(T)_{\ell}}$ has a finite projective resolution. It follows that the cohomology $\R$-modules $\mathcal{H}^i(\M(\F)\otimes^{\operatorname{L}}(\R/\mathfrak{m}_0^n))$ are finite-dimensional over $\Qlbar$. From the short exact sequences
  \begin{multline*}
    0\to \operatorname{R}^1\lim\mathcal{H}^{i-1}(\M(\F)\otimes^{\operatorname{L}}\R/\mathfrak{m}_0^n) \to \\
    \to \mathcal{H}^{i}(\M(\F)^{\wedge}) \to \operatorname{lim}\mathcal{H}^{i}(\M(\F)\otimes^{\operatorname{L}}\R/\mathfrak{m}_0^n) \to 0,
  \end{multline*}
  where $\M(\F)^{\wedge}$ stands for the derived completion with respect to $\mathfrak{m}_0$, we get, for $i < 0$,
  \[
    \mathcal{H}^{i}(\M(\F)^{\wedge}) \simeq \operatorname{lim}\mathcal{H}^{i}(\M(\F)\otimes^{\operatorname{L}}\R/\mathfrak{m}_0^n),
  \]
  since $\operatorname{R}^1\lim$ vanishes for the inverse system of finite-dimensional vector spaces. From $t$-exactness of derived completion on finite modules we get
\[
\operatorname{lim}\mathcal{H}^{i}(\M(\F)\otimes^{\operatorname{L}}\R/\mathfrak{m}_0^n) = 0,\text{ for } i < 0.
\]
Again, since the system in question is of finite-dimensional vector spaces, this implies that it is essentially zero. By property \ref{item:7}, we get that the system of $^p\mathcal{H}^i(\F\star\E_n^{\lambda}[r])$ is essentially zero.

 \ref{item:21} $\implies$ \ref{item:20}. Similarly to the above, we get that, for the restriction of $\M(\F)$ to every connected component, and for any maximal ideal $\mathfrak{m}$ corresponding to a $\Qlbar$-point in this component, the $\mathfrak{m}$-adic completion of $\mathcal{H}^i(\M(\F))$ is 0 for $i < 0$. Since $\R$ is Noetherian, this implies that $\mathcal{H}^i(\M(\F))/\mathfrak{m} = 0$ for all maximal $\mathfrak{m}$, and so $\mathcal{H}^i(\M(\F)) = 0$ by Nakayama's lemma. The result now follows from the property \ref{item:3}.
\end{proof}
\begin{proposition}
\label{sec:mellin-transform-2}
  If for $\F \in P(T)$ we have $^p\mathcal{H}^0(\F \star \mathcal{L}_{\lambda}[r]) = 0$ for all $\lambda \in \mathcal{C}(T)(\Qlbar)$, then $\F = 0$.
\end{proposition}
\begin{proof}
  Similarly to the proof of Proposition \ref{sec:mellin-transform}, we have that (underived) fibers of $\M(\F)$ at all geometric points vanish, hence so does $\M(\F)$. The result now follows from the property \ref{item:3}.
\end{proof}
We will need the following variation of the above propositions.

Let $S \subset \ct$ be a finite subset. We say that $\F \in D^b(T)$ is $S$-monodromic, if it is a direct sum of $\lambda$-monodromic complexes with $\lambda \in S$ (note that there are no non-zero morphisms between monodromic complexes with different monodromies).

\begin{proposition}
  \label{sec:mellin-transform-1}
  Fix a complex $\F \in D^b(T)$ and a finite subset $S \subset \ct$.
\begin{enumerate}[label=\alph*),ref=(\alph*)]
\item \label{sec:mellin-transform-3} The system of $^p\mathcal{H}^i(\F\star\E_n^{\lambda}[r])$ is essentially zero for all $\lambda \in \mathcal{C}(T)(\Qlbar)$ with $\lambda \notin S$, and all $i$, if and only if ${}^p\mathcal{H}^i(\F)$ are shifted $S$-monodromic local systems on $T$.
\item\label{sec:mellin-transform-4} If for $\F \in P(T)$ we have $^p\mathcal{H}^0(\F \star \mathcal{L}_{\lambda}[r]) = 0$ for all $\lambda \in \mathcal{C}(T)(\Qlbar)$ with $\lambda \notin S$, then $\F$ is a shifted $S$-monodromic local system on $T$.
\end{enumerate}
\end{proposition}
\begin{proof}
  It follows from the proof of Proposition \ref{sec:mellin-transform} for part \ref{sec:mellin-transform-3} and Proposition \ref{sec:mellin-transform-2} for part \ref{sec:mellin-transform-4} that $\M(\F)$ is supported at $S \subset \mathcal{C}(T)(\Qlbar)$. The result now follows from property \ref{item:7}. 
\end{proof}
\subsection{Braided structure and the twist on the reductive group.}
\label{sec:conv-reduct-group}
Here we recall the necessary constructions from \cite[Appendix B]{boyarchenkoCharacterSheavesUnipotent2013} and some notations from the introduction. 

Let $G$ be a connected reductive group defined over $\kk$. The category $D^b_G(G),$ where $G$ acts on itself by conjugation, is braided monoidal. We denote its braided structure by \[\beta_{A,B}:A\star B \to B \star A.\]

We also consider the ``underived'' convolution $\star^0$ on $P_G(G)$, defined as
\[
  A\star^0 B = {}^p\mathcal{H}^0(A\star B).
\]
This equips $P_G(G)$ with the structure of a braided monoidal category.

For any $M \in D^b_G(G)$ we have an isomorphism $\theta_M:M\to M$ called a twist (see ibid). It satisfies
\begin{equation}
  \label{eq:4} \theta_{A\star B}=\beta_{B,A}\circ\beta_{A,B}\circ (\theta_A\star\theta_B)
\end{equation}
Let $Z_G(g)$ stand for the centralizer of a point $g\in G(\kk)$. Since $M$ is equivariant with respect to the adjoint action, the stalk $M_g$ of $M$ at $g$ has an action of $Z_G(g)$. It is known that the action of $\theta_M$ on $M_g$ can be identified with the action of $g \in Z_G(g)(\kk).$

\begin{proposition}
  \label{sec:twist-cor} Assume that any $g \in G$ lies in the neutral connected component $Z_G(g)^0\subset Z_G(g)$. Then $P_G(G)$ is symmetric, meaning that
  $$
  \beta_{A,B}\circ \beta_{B,A} = Id_{B \star^0 A}
  $$
  for all $A, B \in P_G(G)$.
\end{proposition}
\begin{proof}
  Since any $g \in G$ lies in $Z_G(g)^0$, automorphism $\theta$ acts by the identity on stalks of any equivariant complex $\F \in D^b_G(G)$. This implies that $\theta$ is the identity on any perverse sheaf, which, by \eqref{eq:4}, gives the result.
\end{proof}
\begin{remark}
  \label{sec:braid-struct-twist}
Note that in \cite{tolmachovFunctorAffineFinite2018} the braided structure on $D^b_{\GL_n}(\GL_n)$ was mistakenly treated as symmetric. The current, more technically involved, paper, thus, closes the gap in \cite{tolmachovFunctorAffineFinite2018}.  We expect, however, that the construction in \cite{tolmachovFunctorAffineFinite2018} can be fixed also by modifying the braided structure on the derived category $D^b_{\GL_n}(\GL_n)$. See Subsection \ref{sec:conj-regard-symm} below.
\end{remark} 
\subsubsection{Conjecture regarding the symmetric structure on the derived category.}
\label{sec:conj-regard-symm} As was mentioned in the introduction, the category $D^b_G(G)$ for an algebraic group $G$ may fail to be symmetric braided even if $G$ itself is commutative.

Let $(\mathcal{C}, \star,\beta)$ be any braided monoidal category that is $k$-linear over a field $k$ of characteristic 0. Under certain conditions on the braided monoidal category $\mathcal{C}$, the techniques of \cite{drinfeldQuasitriangularQuasiHopfAlgebras1991} (see also \cite{etingofQuasisymmetricUnipotentTensor2009}) can be used to modify the braided structure on $\mathcal{C}$ to make it symmetric. For a triple of objects $A,B,C \in \mathcal{C}$ consider the following endomorphisms of $A \star B \star C$:
\begin{align*}
  \nu_{A,B} &= (\beta_{B,A}\circ\beta_{A,B})\star\id_C-\id_{A\star B \star C},\\
  \nu_{B,C} &=\id_A\star\, (\beta_{C,B}\circ\beta_{B,C})-\id_{A\star B \star C}.
\end{align*}
We say that the braided structure on $\mathcal{C}$ is unipotent if the subalgebra in $\operatorname{End}_{\mathcal{C}}(A\star B\star C)$ generated by $\nu_{A,B}$ and $\nu_{B,C}$
is nilpotent for any triple $A, B, C$. We put forward the following conjecture.
\begin{conjecture}
  \label{sec:conj-regard-symm-1}
  The canonical braided structure on $D^b_{\GL_n}(\GL_n)$ is unipotent.
\end{conjecture}

\begin{remark}
  \label{sec:conj-regard-symm-2}
  It is easy to see (Section \ref{sec:conv-reduct-group}), that if any element $g \in G$ lies in the neutral connected component of its centralizer in $G$, we have that $\theta_A - \id_A$ is a nilpotent element in $\operatorname{End}(A)$ for any $A \in D^b_G(G)$. Since all the centralizers of elements of $\GL_n$ are connected, this property holds for $\GL_n$. In this case, the braiding being unipotent is equivalent to the condition that for any triple of objects $A,B,C \in \mathcal{C}$ the subalgebra in $\operatorname{End}_{\mathcal{C}}(A\star B\star C)$ generated by $\theta_{A\star B}\star\id_C-\id_{A\star B\star C}$ and $\id_A\star\,\theta_{B\star C}-\id_{A \star B \star C}$ is nilpotent.  
\end{remark}
\begin{remark}
  One can show that for a torus $T$ over $\kk$, the canonical braiding on $D^b_T(T)$ is unipotent.
\end{remark}
\subsection{Parabolic restriction and Harish-Chandra transform.}
We mostly follow \cite{bezrukavnikovParabolicRestrictionPerverse2021} in our geometric setup.

Let $G$ be a connected reductive group defined over $\kk$, let $P \subset G$ be a parabolic subgroup, $U \subset P$ its unipotent radical and $L$ its Levi subgroup.

Let \[\yy_P = (G/U \times G/U)/L,\] where $L$ acts via the right diagonal action.
Note that $G$ acts on $\yy_P$ via the left diagonal action.

We equip the category $D^b_G(\yy_P)$ with the monoidal structure $\star$ given by $*$-convolution: to define it, we identify $D^b_G(\yy_P)$ with $D^b_{U^2}(G/_{\Ad}L)$, where the subscript $U^2$ stands for the full subcategory of complexes with constant cohomology along $U^2$-orbits, and use the formula \eqref{eq:8}. 

\subsubsection{Harish-Chandra transform.}\label{sec:harish-chandra-trans-1} Consider the diagram
\[
\begin{tikzcd} & G\times G/P \arrow[ld, "p"'] \arrow[rd, "q"] & \\ G &
& \yy_P
\end{tikzcd}
\]
where $p$ is the projection and $q$ is the map $q(g,xP) = (xU,gxU)L$. The functor
\[
  \hc=q_*p^!:D^b_G(G) \to D^b_G(\yy_P),
\]
where the action of $G$ on itself is the adjoint action, is called the Harish-Chandra transform. It is well-known to be monoidal with respect to the $*$-convolution.
It has a left adjoint functor
\[
  \chi = p_!q^* \simeq p_*q^![-2\dim U].
\]
More generally, for parabolics $P \subset Q$, we have functors $\hc_{P}^Q, \chi_{P}^Q$, defined in \cite{lusztigParabolicCharacterSheaves2004}. We recall the definition. Let
\[\tilde{Z}_{P, Q} = \{(xU_Q, yU_Q, zU_P) \in G/U_Q \times G/U_Q \times G/U_P: x^{-1}z \in Q\},\] and let $Z_{P,Q} = \tilde{Z}_{P,Q}/(L_P \times L_Q)$, where $L_P$ acts on the $G/U_P$-factor on the right, and $L_Q$ acts diagonally on the right on the $G/U_Q \times G/U_Q$-factor. We have maps
     \[
       \begin{tikzcd} & Z_{P,Q} \arrow[ld, "{f_{P,Q}}"'] \arrow[rd,
"{g_{P,Q}}"] & \\ \mathcal{Y}_Q & & \mathcal{Y}_P
\end{tikzcd}
     \]
     given by
     \[
       f_{P,Q}(xU_Q, yU_Q, zU_P) = (xU_Q, yU_Q),
     \]
     \[
       g_{P,Q}(xU_Q, yU_Q, zU_P) = (zU_P, yx^{-1}zU_P).
     \]
     Note that $g_{P,Q}$ is well-defined because $x^{-1}z \in Q$, while $Q$ normalizes $U_Q$, and $U_Q \subset U_P$, since $P \subset Q$. 
     Write
     \begin{align*}
       \chi_{P}^{Q} &= f_{P,Q*}g_{P,Q}^![-2\dim Q/P]: D^b(\yy_P/G) \to D^b(\yy_Q/G),\\
       \hc_P^Q &=g_{P,Q*}f_{P,Q}^!: D^b(\yy_Q/G) \to D^b(\yy_P/G).
     \end{align*}
   In our previous notations, we have $\hc = \hc_P^G, \chi = \chi_P^G$ (note that ${\yy_G \simeq G}$ via the map $(x,y) \mapsto yx^{-1}$). It was proved in ibid.\@ that, if $P \subset Q \subset R$, we have
     \[
       \chi_Q^R\chi_P^Q = \chi_{P}^R\text{ and } \hc^Q_P\hc^R_Q = \hc^R_P.
     \]
     The functors $\hc_{P}^Q$ are monoidal with respect to convolution.
     
Let $\Delta_P$ be the closed $G$-orbit in $G/P \times G/P$ and let $\tilde{\Delta}_P$ be its preimage in $\yy_P$. Let \[\iota_P:\tilde{\Delta}_P\to\yy_P\] be the corresponding embedding.
Identify \[P\backslash L = G\backslash\tilde{\Delta}_P\] via the map $t \mapsto (t, tU)$. Here $P$ acts on $L$ through the adjoint action of $L$ on itself. Abusing notation, we will write $\iota_P^!$ for the composed functor \[ D^b_G(\yy_P) \xrightarrow{\iota_P^!} D^b_G(\tilde{\Delta}_P) = D^b_P(L) \to D^b_L(L), \]  where the last arrow is the forgetful functor. Let $\iota_{P!}$ stand for its left adjoint.

We have the following description of the parabolic restriction functor $\operatorname{Res}_{P}^G$ and parabolic induction functor $\operatorname{Ind}_{P}^G$
(cf. \cite[Remark 4.2]{bezrukavnikovParabolicRestrictionPerverse2021}):
\begin{align}
  \label{eq:6} \operatorname{Res}_{P}^G
  &= \iota_P^!\hc_*,\\
  \label{eq:5}\operatorname{Ind}_{P}^G
  &= \chi_P^G\iota_{P!}.
\end{align}
We will also need the following proposition, parallel to \cite[Lemma 6.6]{lusztigParabolicCharacterSheaves2004} and, for $Q = G$, \cite[Lemma 8.5.4]{ginzburgAdmissibleModulesSymmetric1989}. For a pair of parabolics $P \subset Q$ with Levi subgroups $L$ and $M$, respectively, define
\[
  \mathcal{N}_{P}^Q = \{(x, gP)\in M \times Q/P: x \in {}^gU_P\}.
\]
Let $\pi: \mathcal{N}_P^Q \to M$ be the natural projection, and define \[\Spgr^Q_P = \pi_*\cc_{\mathcal{N}_P}[2 \dim U_P \cap M].\]
\begin{proposition}
  \label{sec:harish-chandra-trans}
  We have an isomorphism of functors
  \[
    \chi_P^Q\hc^Q_P(-) \simeq \iota_{Q*}\Spgr^Q_P \star (-). 
  \]
  In particular, the identity functor is a direct summand of $\chi_P^Q\hc^Q_P$.
\end{proposition}
We sketch the proof, parallel to the proof in ibid.
\begin{proof}
  To unburden the notations, we omit the quotients by the left diagonal $G$-action everywhere. All arrows in all diagrams are $G$-equivariant. By base change, the left-hand side is naturally isomorphic to the functor $p_*q^!$ defined by the diagram
 \[
   \begin{tikzcd} & X \arrow[ld,"p"'] \arrow[rd,"q"] & \\ \yy_Q & & \yy_Q \end{tikzcd}
\]
where
\begin{multline*}
  X= \{((x_1U_Q,x_2U_Q)M,(x_3U_Q,x_4U_Q)M,zP) \in \yy_Q^2\times G/P, \\
  x_1^{-1}z \in Q, x_3^{-1}z \in Q, x_1x_2^{-1}x_4x_3^{-1} \in {}^zU_P\},
\end{multline*}
with
\begin{align*}
  q((x_1U_Q,x_2U_Q)M,(x_3U_Q,x_4U_Q)M,zP) &= (x_1U_Q, x_2U_Q)M, \\
  p((x_1U_Q,x_2U_Q)M,(x_3U_Q,x_4U_Q)M,zP) &= (x_3U_Q, x_4U_Q)M.
\end{align*}
Note that the condition $x_1x_2^{-1}x_4x_3^{-1} \in {}^zU_P$ in the definition of $X$ can be rewritten as:
\[
  z^{-1}x_1(x_2^{-1}x_4)x_3^{-1}z \in U_P.
\]
Since $U_P \subset Q$, this and the two conditions $x_1^{-1}z \in Q$ and $x_3^{-1}z \in Q$ imply that $x_2^{-1}x_4 \in Q$. We will use this observation in the computations below.

  Similarly, the right-hand side is naturally isomorphic to the functor $p'_*q'^!$ defined by the diagram
 \[
   \begin{tikzcd} & X' \arrow[ld,"p'"'] \arrow[rd,"q'"] & \\ \yy_Q & & \yy_Q 
\end{tikzcd}
\]
where
\begin{multline*}
  X'= \{((aU_Q,bU_Q,cU_Q)M,zP) \in (G/U_Q)^3/M \times G/P, \\ a^{-1}z \in Q, ba^{-1} \in {}^zU_P\},
\end{multline*}
with
\begin{align*}
  q'((aU_Q,bU_Q,cU_Q)M,zP) &= (bU_Q, cU_Q)M\\
  p'((aU_Q,bU_Q,cU_Q)M,zP) &= (aU_Q, cU_Q)M.
\end{align*}
We have a commutative diagram
 \[
   \begin{tikzcd}
     & X' \arrow[ld,"q'"'] \arrow[rd,"p'"] & & X \arrow[ll, "\pi"] \arrow[ld,"p"'] \arrow[rd,"q"] & \\ \yy_Q & & \yy_Q & & \yy_Q
\end{tikzcd}
\]
with
  \begin{multline*}
    \pi((x_1U_Q,x_2U_Q)M,(x_3U_Q,x_4U_Q)M,zP) =\\ ((x_3x_4^{-1}x_2U_Q, x_1U_Q, x_2U_Q)M, zP). 
  \end{multline*}
  It is easy to see that $\pi$ is an isomorphism with the inverse given by the following map $X' \to X$:
  \begin{align*}
    ((aU_Q, bU_Q, cU_Q)M, zP) \mapsto ((bU_Q, cU_Q)M, (aU_Q, cU_Q)M, zP),
\end{align*}
  which implies the result.
\end{proof}
\subsubsection{Pro-unit in the category of character sheaves.} Assume that $G$ has connected center.
 Let $I$ be the set of the simple roots of $G$ corresponding to $B$. For a parabolic subgroup $P \supset B$ with a Levi $L$ indexed by a subset $J \subset I$, let $W_J$ be the corresponding parabolic subgroup of $W$. Let $W^J$ stand for the set of minimal length representatives of $W/W_J$. 

We say that $\F \in D^b_G(\yy_P)$ has character $\lambda \in \mathcal{C}(T)(\Qlbar)$, if $\hc_B^P(\F)$ is a monodromic complex with monodromy in the $W_J$-orbit of ${\lambda}$. We write $D^\lambda(\yy_P)$ for the full subcategory consisting of complexes with character $\lambda$. Note that, by definition, $D^\lambda(\yy_P) = D^{w\lambda}(\yy_P)$ for $w \in W_J.$ Let $P^{\lambda}(\yy_P)$ stand for the full subcategory of perverse objects in $D^\lambda(\yy_P)$. It is well-known that if $\mu\notin W\lambda$, then for all $X \in D^{\lambda}(\yy_P)$ and $Y \in D^{\mu}(\yy_P)$ we have $\Hom(X,Y) = 0$. 

For $P=G,$ where we have $\yy_G \simeq G,$ objects in $P^{\lambda}(G)$ are called character sheaves with character $\lambda$. More generally, for any finite subset $S \subset \cC(T)(\Qlbar)$, we say that objects of $\bigoplus_{\lambda \in S}P^{\lambda}(G)$, where the latter is considered as a full subcategory of $P_G(G),$ are character sheaves with characters in $S$.

For $\lambda \in \mathcal{C}(T)(\Qlbar)$ let $W_\lambda$ be its stabilizer in $W$. Since $G$ is assumed to have connected center, $W_\lambda$ is a Coxeter group conjugate to a parabolic subgroup of $G$. Let $\mathfrak{m}_{\lambda+} \subset \mathfrak{m}_0$ be the ideal generated by $W_\lambda$-invariant polynomials in $\mathfrak{m}_{0}$. Let $\mathfrak{E}_n^{\lambda}$ be the perverse local system on $T$ satisfying $\mathcal{M}(\mathfrak{E}_n^{\lambda}) = m_{\lambda}^*(\mathfrak{R}/\mathfrak{m}_{\lambda+}^n)$. Abusing notation, write $\mathfrak{E}_n^{\lambda}$ for the sheaf on $\tilde{\Delta}_B$ corresponding to $\mathfrak{E}_n^{\lambda} \in D^b(T)$. We will need the following
\begin{proposition}
\label{centralunit} For any $\lambda \in \mathcal{C}(T)$, there is an inverse system of perverse sheaves ${E}_n^{\lambda} \in P_G(G)$ such that \[\hc_*(E_n^{\lambda}) = \operatorname{Res}_{B}^G(E_n^{\lambda})= \bigoplus_{\mu\in W\lambda}\mathfrak{E}_n^{\mu}.\]
\end{proposition}
\begin{proof} This follows from Theorem 7.1 together with \cite[Proposition 3.2]{chenConjectureBravermanKazhdan2022} (see also \cite[Section 5.3]{hhh}).
\end{proof}

This allows us to formulate the following analogue of Proposition \ref{sec:mellin-transform-1} for categories of character sheaves on $\GL_n$. 
\begin{proposition}
  \label{sec:pro-unit-character-1}
Let $G = \mathrm{GL}_n$. Let $S$ be a finite $W$-invariant set of points in $\ct$.  Write $({}^pD^{\leq 0}, {}^pD^{\geq 0})$ for the perverse $t$-structure on $D^b_G(G)$.
  \begin{enumerate}[label=\alph*),ref=(\alph*)]
  \item \label{item:8} For $\F \in D^b_G(G)$ we have $\F \in {{}^pD^{\geq 0}}$ if and only if the system of $^p\mathcal{H}^i(\F\star E_n^{\lambda})$ is essentially zero for all $\lambda \in \mathcal{C}(T)(\Qlbar)$ and all $i < 0$.
\item \label{item:1} The system of $^p\mathcal{H}^i(\F\star E_n^{\lambda})$ is essentially zero for all $\lambda \in \mathcal{C}(T)(\Qlbar)$ with $\lambda \notin S$, and all $i$, if and only if the ${}^p\mathcal{H}^i(\F)$ are character sheaves with characters in $S$.
\item \label{item:2}  If for $\F \in P_G(G)$ we have $^p\mathcal{H}^0(\F \star E_1^{\lambda}) = 0$ for all $\lambda \in \mathcal{C}(T)(\Qlbar)$ with $\lambda \notin S$, then $\F$ is a character sheaf with characters in $S$.
\end{enumerate}
\end{proposition}
\begin{proof} From \eqref{eq:6} and Proposition \ref{centralunit} it is clear that
  \begin{equation}
    \label{eq:7} \operatorname{Res}_{B}^G(\F\star E_n^{\lambda})
= \operatorname{Res}_{B}^G(\F)\star \left(\bigoplus_{\mu\in
W\lambda}\mathfrak{E}_n^{\mu}\right).
\end{equation}
By the result of \cite{bezrukavnikovParabolicRestrictionPerverse2021}, the parabolic restriction functor is $t$-exact as a functor
  \[ D^b_G(G)\to D^b(T).
  \]
Since there are no non-trivial cuspidal pairs for $D^b_G(G)$ (see \cite{lusztigIntersectionCohomologyComplexes1984}), 
\[\operatorname{Res}_{B}^G:P_G(G)\to P(T)\] is a conservative functor. It follows that it is a faithful functor between the abelian categories above, and so \ref{item:8} follows.

For \ref{item:1} and \ref{item:2}, by Proposition \ref{sec:mellin-transform-1}, we get that
\[
  {}^p\mathcal{H}^i(\operatorname{Res}_{B}^G(\F)) = \operatorname{Res}_{B}^G({}^p\mathcal{H}^i(\F))
\]
are shifted $S$-monodromic local systems. For a shifted $S$-monodromic local system $\E$ on $T$, $\operatorname{Ind}_{B}^G(\mathcal{E})$ is a character sheaf with character in $S$. Since $\operatorname{Res}^G_B$ is exact and conservative, $\F$ is a subobject of $\operatorname{Ind}^G_B\operatorname{Res}^G_B(\F)$, and we get the result.
\end{proof}

\begin{corollary}
\label{sec:pro-unit-character} In the notations of Proposition \ref{sec:pro-unit-character-1}, if a morphism $f:\mathcal{A} \to \mathcal{B},$ where $ \mathcal{A},\mathcal{B}$ are perverse in $D^b_G(G)$, is such that \[f\star^0 \id_{E_n^{\lambda}}:\mathcal{A}\star^0 E_n^{\lambda} \to \mathcal{B}\star^0 E_n^{\lambda}\] is an isomorphism for all $n$ and $\lambda \in \mathcal{C}(T)(\Qlbar)$ with $\lambda \notin S$, the perverse cohomology sheaves of the cone $C_f$ of $f$ are character sheaves with characters in $S$.
\end{corollary}
\begin{proof} The long exact sequence of perverse cohomology for the distinguished triangle
  \[
    \mathcal{A} \star E_n^{\lambda} \rightarrow \mathcal{B}  \star E_n^{\lambda} \to C_f\star E_n^{\lambda} \to \mathcal{A}\star E_n^{\lambda}[1]
\]
reads
\begin{multline}
  \dots \to {}^p\mathcal{H}^{-1}(\mathcal{A}\star E_n^{\lambda}) \to {}^p\mathcal{H}^{-1}(\mathcal{B}\star E_n^{\lambda}) \xrightarrow{c_\lambda} {}^p\mathcal{H}^{-1}(C_f\star E_n^{\lambda}) \xrightarrow{0}\\
  \xrightarrow{0}\mathcal{A}\star^0 E_n^{\lambda}\to \mathcal{B}\star^0 E_n^{\lambda}\xrightarrow{0} {}^p\mathcal{H}^0(C_f)\star^0 E_n^{\lambda}\to 0.
\end{multline}
We get that ${}^p\mathcal{H}^0(C_f)\star^0 E_n^{\lambda} = 0$ for $\lambda\notin S$, and so, by Proposition \ref{sec:pro-unit-character-1} \ref{item:2}, ${}^p\mathcal{H}^0(C_f)$ is a character sheaf with characters in $S$. It follows that ${}^p\mathcal{H}^0(C_f)\star E_n^{\lambda} = 0$ for $\lambda \notin S$, and so ${}^p\mathcal{H}^{-1}(C_f\star E_n^{\lambda}) = {}^p\mathcal{H}^{-1}(C_f)\star^0 E_n^{\lambda}$. Since maps $c_{\lambda}$ are surjective, we get from Proposition \ref{sec:pro-unit-character-1} \ref{item:8} applied to $\F = \mathcal{B}$ that the system of ${}^p\mathcal{H}^{-1}(C_f)\star^0 E_n^{\lambda}$ is essentially zero for all $n$ whenever $\lambda \notin S$, and so, by \ref{sec:pro-unit-character-1} \ref{item:1}, ${}^p\mathcal{H}^{-1}(C_f)$ is a character sheaf with characters in $S$. 
\end{proof}
\subsubsection{Generic monodromy.} We assume again that $G$ has connected center.
  For a complex $X \in D^b_G(\yy_P)$ such that $\hc_P(X)$ is monodromic, write $\Pi_{\lambda}X$ for the summand of $X$ with character $\lambda$. By Proposition \ref{centralunit}, $\Pi_{\lambda}$ can be explicitly written as $\invlim{n}(E_n^{\lambda,L}\star -) $, where $E_n^{\lambda,L}$ denotes the sheaf constructed in Proposition \ref{centralunit} for $G=L$, considered as a sheaf on $\yy_P$ via the functor $\iota_{P!}$.

For objects $X, C_1, \dots, C_k$ of a triangulated category, we write $X \in \langle C_1, C_2, ..., C_k\rangle$ if there exists a sequence of objects $X_1, \dots, X_k,$ with $X_1 = C_1, X_k = X$, such that for all $i < k$ $(X_{i-1}, X_i, C_{i})$ is a distinguished triangle.

The following proposition is a straightforward consequence of standard distinguished triangles in the constructible derived categories.
\begin{proposition}
  \label{shrieck_filtr} Let $X$ be a variety stratified by locally closed subvarieties $\{S_t\}_{t=1}^n$, and let $j_t: S_t \to X$ be the corresponding locally-closed embeddings. Assume that $S_k \subset \overline{S}_l$ implies $k < l$. Then for every $\mathcal{F} \in D^b(X)$ we have $\mathcal{F} \in \langle j_{t*}j^!_t\mathcal{F}\rangle_{t=1}^n$.
\end{proposition}
\begin{proof}
  See for example the proof of Lemma 6.2.5 in \cite{hhh}.
\end{proof}

We say that a character $\lambda \in \mathcal{C}(T)(\Qlbar)$ is $P$-generic, if its stabilizer in the full Weyl group $W$ is equal to $W_J$.

\begin{proposition}
  \label{sec:generic-monodromy}
  Let $\lambda \in \ct$ be $P$-generic for a parabolic subgroup $P \supset B$.   We have the following isomorphism of functors
  \[
    \Pi_{\lambda}\hc_{P*} = \Pi_\lambda \iota_{P!} \res^G_{P}.
\]
\end{proposition}

 \begin{proof} 
  By~\eqref{eq:6} we have $\iota_{P!} \res^G_{P}(\mathcal{F}) = \iota_{P!}\iota_{P}^!\hc_{P*}(\mathcal{F})$. Let \[j: \yy_P \backslash \tilde{\Delta}_P \to \yy_P\] be the complementary open embedding. We have a distinguished triangle
\[
  (\iota_{P!}  \res^G_L(\mathcal{F}), \hc_P(\mathcal{F}), j_*j^*\hc_P(\mathcal{F}) ).
\]
  We will now prove that $\Pi_{\lambda}j_*j^*\hc_P(\mathcal{F})= 0$. We have
\[
  \hc^P_B(j_*j^*\hc_P(\mathcal{F})) = j'_*\mathcal{G},
\]
where $j'$ is the embedding of an open subset \[V = \{(xU,yU)T \in \yy, x^{-1}y \notin P\},\] and $\mathcal{G}\in D^b_G(V)$. Let $\tilde{V}$ be the preimage of $V$ in $G/U \times G/U$. $\tilde{V}$ has a stratification by locally closed subsets \[\tilde{V}_w = \{(xU, yU), x^{-1}y \in Ug_wP\},\] where $w$ runs through minimal length representatives of non-unital congruence classes in $W/W_J$, and $g_w$ are some lifts of $w$ to $N_G(T)$. Let $\tilde{j}'_{w}:\tilde{V}_w \to \tilde{V}$ be the corresponding locally closed embeddings. Consider $j'_*\mathcal{G}$ as a $T$-equivariant complex on $G/U \times G/U$. By Proposition \ref{shrieck_filtr}, $j_*'\mathcal{G} \in \langle \tilde{j}'_{w*}\tilde{j}_w'^!j_*'\mathcal{G}\rangle_{w}$, so it is enough to prove that $\Pi_{\lambda}\tilde{j}'_{w*}\tilde{j}_w'^!j_*'\mathcal{G} = 0$. Consider a point $(xU, yU) \in \tilde{V}_{w}$. Choose $x, y$ in the corresponding congruence class so that $x^{-1}y = g_wl$, $l \in L.$ Since $\lambda$ is $P$-generic, there is a coweight $\mu^{\vee}:\mathbb{G}_m \to T$ with $\langle\lambda, \mu^{\vee} - w(\mu^{\vee})\rangle \neq 0$, and such that $\mu^{\vee}(\mathbb{G}_m)$ is in the center of $L$. Then, for $t \in \mathbb{G}_m$, we have
\[
  \textrm{Ad}(\mu^{\vee}(t))(x^{-1}y) = \mu^{\vee}(t)g_wl\mu^{\vee}(t)^{-1} = \mu^{\vee}(t)w(\mu^{\vee}(t))^{-1}(x^{-1}y),
\]
and so a $G$-equivariant sheaf on $\tilde{V}_m$ that is equivariant with respect to the right diagonal $T$-action, is also equivariant with respect to the action of $\mathbb{G}_m$ defined by the coweight $\mu^{\vee} - w(\mu^{\vee})$ on one of the components. Since a non-zero equivariant complex can not have non-trivial monodromy, the result follows.
\end{proof}
\begin{corollary}
  \label{sec:generic-monodromy-1}
 Assume that ${\lambda}\in \ct$ is $P$-generic for a parabolic $P$. Restricted to the category $P^{\lambda}(G)$, the functor \[\Pi_\lambda\res_{P}^G:P^{\lambda}(G) \to P^{\lambda}(L)\] is a braided monoidal equivalence of categories. 
\end{corollary}
\begin{proof}

We first show that $\Pi_\lambda\res_P^G$ is an equivalence. We claim that the inverse functor is given by $\chi_P^G$. By base change, it is easy to see that, for any $\F \in D^b(\yy_P)$,
\[
\hc^P_B\hc_{P}^G\chi_P^G\F \in \langle\Phi_w(\F)[d_w]\rangle_{w \in W^J},
\]
where $d_w$ are some integral homological shifts, $d_1 = 0$, and $\Phi_w = p_*q^!$ is defined from the diagram
\[
   \begin{tikzcd} & \yy_w^P \arrow[ld,"p"'] \arrow[rd,"q"] & \\ \yy_B & & \yy_B 
\end{tikzcd}
\]
where
\[
  \yy_w^P = \{(g, xB, yB) \in  G \times G/B \times G/B: (xB, yB) \in \cO_w\},
\]
and $p(g,xB,yB) = (xU,gxU), p(g,xB,yB) = (yU,gyU)$. If $\F$ is mo\-nodromic with monodromy $\lambda$, $\Phi_w(\F)$ is monodromic with monodromy $w\lambda$. It follows that, for $\F \in D^{\lambda}(L)$,
\[
  \hc_B^P\Pi_{\lambda}\hc_P^G\chi_P^G \F \simeq \Pi_\lambda\hc_B^P\hc_P^G\chi_P^G\F \simeq \hc_B^P\F,
\]
and so
\(  \Pi_{\lambda}\hc_P^G\chi_P^G \F \simeq \F.
\)
It follows that, for any $\F \in D^{\lambda}(L)$,
\[
  \chi_P^G\Pi_\lambda\hc_P^G\chi_P^G\F \simeq \chi_P^G\F.
\]

From Proposition \ref{sec:harish-chandra-trans}, to show that $\chi_P^G$ is left inverse to $\Pi_{\lambda}\hc_P^G$, it is enough to show that any irreducible object $\mathcal{G} \in P^{\lambda}(G)$ is obtained as a summand of $\chi_P^G$ for some $\F \in P^{\lambda}(L)$, so that the corresponding morphism
\[
  \mathcal{G} \to \chi_P^G\Pi_\lambda\hc_P^G\mathcal{G}
\]
is an isomorphism. This follows from \cite[\S{17.12}]{lusztigCharacterSheavesIV1986}.

 Since monodromic sheaves with different monodromies are orthogonal with respect to convolution, it follows from Proposition \ref{sec:generic-monodromy} that $\Pi_{\lambda}\iota_{P!}\res^G_{P}$ is monoidal, since $\hc_{P*}$ and $\Pi_{\lambda}$ are monoidal.

The functor $\hc_*$ is equipped with the central structure (see \cite{ginzburgAdmissibleModulesSymmetric1989}, \cite{bezrukavnikovCharacterDmodulesDrinfeld2012}). It is easy to see that, for complexes $\F$ with $\hc_*(\F)$ supported on $\tilde{\Delta}_P$, this central structure restricts to the braiding structure coming from $P$-equivariance, and so $\Pi_\lambda\res_{P}^G$ respect the braiding. 
\end{proof}
\subsection{Exterior powers of the Springer sheaf.} We recall some of the notations from the introduction. Let $P$ again be an arbitrary parabolic. Define
\[
  \mathcal{N}_{P} = \{(x, gP)\in G\times G/P: x \in {}^gU\}.
\]

Let $\pi: \mathcal{N}_P \to G$ be the natural projection, and
define $\Spgr = \pi_*\cc_{\mathcal{N}_P}[2 \dim U].$

Let now $G = \GL_n = \GL(V)$, and $W=S_n$. We keep the notation of Section \ref{sec:action-repg}, writing $P = P_1$ for convenience. We consider 
\[ \Spgr_n := \Spgr_{P}
\] as an object in $P_{\GL_n}(\GL_n).$ By Proposition \ref{sec:twist-cor}, $P_{\GL_n}(\GL_n)$ is a symmetric monoidal category. Our main result is 
\begin{theorem}
  \label{sec:lie-group-case}
  Exterior powers of $\Spgr_n$ with respect to the symmetric monoidal structure $\star^0$ satisfy the following: for $k \geq 0,$ there is an isomorphism
  \[\wedge^k\Spgr_n \simeq \Spgr_{\wedge^k}.\]
\end{theorem} The proof of this theorem occupies the rest of the section.
It is by induction in $n$ and $k$. For $n = 1$ the statement is obvious. From \eqref{eq:6}, for any $\F \in D^b_G(G)$, we have a natural transformation \[ \tilde{\Phi}_n: \operatorname{Ind}^G_{P}\operatorname{Res}^G_{P}(-)\to\Spgr_n \star (-) , \] coming from the isomorphism \[\Spgr_n\star (-) = \chi_P^G \hc_P^G(-),\] and the natural transformation  \[ \iota_{P!}\iota_P^!\hc_P^G(-) \to \hc_P^G(-). \] Assume that the statement of the Theorem \ref{sec:lie-group-case} is known for $n$ and $k$. Define a morphism \[ \Phi_n^k: \Spgr_{\wedge^{k+1}}\to \wedge^{k+1}\Spgr_n , \] as the composition
\[ \Spgr_{\wedge^{k+1}}\xrightarrow{\alpha}\operatorname{Ind}^G_{P}\operatorname{Res}^G_{P}(\Spgr_{\wedge^{k}})\xrightarrow{\tilde{\Phi}_n(\Spgr_{\wedge^k})}\Spgr_n\star^0\wedge^{k}\Spgr_n \xrightarrow{\beta} \wedge^{k+1}\Spgr_n ,
\]
where $\alpha$ (resp. $\beta$) is the inclusion from (resp. the projection to) the corresponding direct summand. We have the following key
\begin{proposition}
  \label{sec:exter-powers-spring}
  Assume that Theorem \ref{sec:lie-group-case} holds for all $\mathrm{GL}_{n'}, n' < n,$ so that the $\Phi_{n'}^{k'}$ are isomorphisms for all $k'$, and the same holds also for $\mathrm{GL}_n$, fixed $k$ and $\Phi_n^k$. Assume also that, for $\lambda \in \mathcal{C}(T)(\Qlbar)$, we have $\operatorname{Stab}_W\lambda \neq W$. Then
\[\Phi_n^k\star \id_{E_{m}^{{\lambda}}}: \Spgr_{\wedge^{k+1}}\star^0
  E_m^{\lambda}\to\wedge^{k+1}\Spgr_n\star^0 E_{m}^{{\lambda}}
\] is an isomorphism for all $m$.
\end{proposition}
Before proving Proposition \ref{sec:exter-powers-spring}, we state the following proposition that is to be used repeatedly.
\begin{proposition}
\label{sec:exter-powers-spring-6}
   Assume that $A\in D^b_G(G)$ satisfies $\hc_*(A) \simeq \operatorname{Res}_B^G(A)$. Then there is a canonical natural isomorphism of functors
    \[ \res_P^G(-\star A)\simeq \res_P^G(-)\star\res_P^G(A),
    \] 
\end{proposition}
\begin{proof}
  Under the assumption of the Proposition, we have that
  \[
    \chi_B^P\hc^G_B(A) \simeq \chi_B^P\res^G_B(A)
  \]
  is supported on $\tilde{\cO}_1$ (see \eqref{eq:6} for notations). By Proposition \ref{sec:harish-chandra-trans},
  \[
    \chi_B^P\hc^G_B(A) = \chi_B^P\hc^P_B\hc^G_P(A)
  \]
  contains $\hc_P^G(A)$ as a direct summand, and so $\hc_P^G(A)$ is supported on $\tilde{\cO}_1$. It follows that the canonical morphism
  \[
    \res^G_P(A) \to \hc^G_P(A)
  \]
  is an isomorphism. The proposition now follows from the base change isomorphism. 
\end{proof}
We will also use the independence of the functor $\res^G_P:P_G(\mathcal{N}_G)\to P_G(\mathcal{N}_L)$ of the choice of a parabolic $P$ with the given Levi $L$, which is proved in \cite[Theorem 7.3]{mirkovicCharacterSheavesReductive2004} in the setting of sheaves on the nilpotent cone $\mathcal{N}_\g$ of $\g = \operatorname{Lie} G$. Since, for a good prime $p$, $\mathcal{N}_\g$ can be $G$-equivariantly identified with $\mathcal{N}_G$ (see \cite[Chapter 6.20]{humphreysConjugacyClassesSemisimple2011} and references therein), the same is true for equivariant perverse sheaves on $\mathcal{N}_G$. In this case we write $\res^G_P =: \res^G_L$.

\begin{proof}[Proof of Proposition \ref{sec:exter-powers-spring}]
Since the statement of the Proposition only depends on the $W$-orbit of $\lambda$, we may assume that the stabilizer of $\lambda$ is some proper parabolic subgroup $W_Q \subset W$ corresponding to a proper parabolic $Q$ of $\GL_n$ with Levi $M$. We will now deduce the statement of the Proposition from the inductive assumption and the fact that \[\mathcal{N}_M = \prod\limits_{k}\mathcal{N}_{\GL_k},\]
  with $k < n$.

    By Corollary \ref{sec:generic-monodromy-1}, it is enough to show that 
  \(\Pi_\lambda\res_{Q}^G(\Phi_n^k\star \id_{E_m^{\lambda}}) \)
  is an isomorphism for all $m \geq 1$.
  This morphism is, by definition, a direct summand of the morphism
   \begin{equation}
     \label{eq:9}
  \operatorname{Ind}^G_{P}\operatorname{Res}^G_{P}(\Spgr_{\wedge^{k}})\star^0 E_m^{\lambda}\xrightarrow{\tilde{\Phi}_n(\Spgr_{\wedge^k})} \Spgr_n\star^0\wedge^{k}\Spgr_n\star^0 E_m^{\lambda}.
 \end{equation} 

  Choose a set of representatives $\Sigma$ of the double cosets $Q\backslash G/P$. For $x \in G(\kk)$ and a subgroup $H \subset G$, write ${}^xH = xHx^{-1}$ and $H^x = x^{-1}Hx$. We assume that $\Sigma$ is chosen so that $L$ and all $M^x$ with $x\in\Sigma$ have a common maximal torus.
  On the left-hand side of \eqref{eq:9}, by the Mackey formula \cite[Proposition 10.1.2]{marsCharacterSheaves1989} and Proposition \ref{sec:exter-powers-spring-6}, we get 
  \begin{align*}
   \Pi_\lambda \res^G_Q&\left(\ind_P^G\res^G_P(\Spgr_{\wedge^k})\star^0E_m^{\lambda}\right) = \\
    &= \bigoplus_{x \in \Sigma}\ind^M_{M\cap{}^xP}{}^x(\res_{L\cap Q^x}^L\res^G_P\Spgr_{\wedge^k}) \star^0\Pi_\lambda \res^G_QE_m^{\lambda} =\\
    &= \bigoplus_{x \in \Sigma}\ind^M_{M\cap{}^xL}{}^x(\res_{M^x\cap L}^G\Spgr_{\wedge^k}) \star^0\Pi_\lambda \res^G_QE_m^{\lambda} =\\
    &=  \bigoplus_{x \in \Sigma}\ind^M_{M\cap{}^xL}{}^x(\res_{M^x\cap L}^{M^x}\res^G_{M^x}\Spgr_{\wedge^k}) \star^0\Pi_\lambda \res^G_QE_m^{\lambda} = \\
    &= \bigoplus_{x \in \Sigma}\ind^M_{M\cap{}^xL}\res_{M\cap {}^xL}^{M}{}^x(\res^G_{M^x}\Spgr_{\wedge^k}) \star^0\Pi_\lambda \res^G_QE_m^{\lambda} = \\
    &= \bigoplus_{x \in \Sigma}\ind^M_{M\cap{}^xL}\res_{M {}^xL}^{M}\res^G_{M}\Spgr_{\wedge^k} \star^0\Pi_\lambda \res^G_QE_m^{\lambda}. 
  \end{align*}
Note that since $P^\lambda(\yy_Q)$ is a monoidal ideal containing $\Pi_\lambda \res^G_QE_m^\lambda$, we can omit $\Pi_\lambda$ on the first $\star^0$-factor from the notation.

Denote
  \begin{align*}
    \Spgr_\lambda := \Pi_\lambda\res_{Q}^G(\Spgr_n\star^0 E_m^{\lambda}) &= \res_{Q}^G(\Spgr_n)\star^0\Pi_\lambda  \res^G_QE_m^{\lambda}  \\
    &= \bigoplus_{x \in \Sigma}\ind^M_{M\cap{}^xL}(\delta_e^{M\cap {}^xL}) \star^0\Pi_\lambda  \res^G_QE_m^{\lambda}.
  \end{align*}
 Here, for any Levi $N$, $\delta^N_e$ is the rank 1 skyscraper sheaf supported at the unit of $N$. Above we used the fact that $\Spgr = \ind_P^G\delta^L_e$ and the Mackey formula.
 
   By Corollary \ref{sec:generic-monodromy-1}, on the right-hand side of \eqref{eq:9} we have

  \[
  \Pi_\lambda\res_{Q}^G\left(\Spgr_n\star^0\wedge^{k}\Spgr_n\star^0 E_m^{\lambda}\right) = \Spgr_\lambda\star^0\wedge^{k}\left(\Spgr_\lambda\right).
\]
Here we used the fact that $E^{\lambda}_m$ is an idempotent for $\star^0$, which follows from the corresponding fact for $G = T$.

For any $x \in \Sigma$, consider the morphism
\[
\begin{tikzcd}
  \displaystyle{\bigoplus\limits_{x\in\Sigma}}\ind^M_{M\cap{}^xP}\res_{{}^xL\cap Q}^{M}\res^G_{Q}\Spgr_{\wedge^k} \star^0\Pi_\lambda \res^G_QE_m^{\lambda} \arrow[d, "\tilde{\Phi}_{n,\lambda}^k"] \\
  \mathcal{S}_\lambda \star^0 \wedge^k\mathcal{S}_\lambda\star^0\Pi_\lambda \res^G_QE_m^{\lambda}
\end{tikzcd}
\]
By the inductive assumption that $\Phi_{n'}^k$  is an isomorphism for all $n' < n$, the restriction
\[
  \Phi_{n,\lambda}^k= \Pi_\lambda\res_{Q}^G(\Phi_n^k\star \id_{E_m^{\lambda}}):\res^G_Q\mathcal{S}_{\wedge^{k+1}}\star^0\Pi_\lambda \res^G_QE_m^{\lambda} \to \wedge^{k+1}\mathcal{S}_\lambda
\]
of \(\tilde{\Phi}_{n,\lambda}^k\) to the corresponding direct summands is also an isomorphism.
\end{proof}
Let $Z(H)$ stand for the center of an algebraic group $H$. For $c \in \mathcal{C}(Z(\GL_n))(\Qlbar)$, let $e_c = \mathcal{E}_1^{c}[1]$ (in the notation of Section \ref{sec:mellin-transform} applied to $G = \mathbb{G}_m$). We will consider $e_c$ as an object in $P_{\GL_n}(\GL_n)$ via the closed embedding $Z(\GL_n) \to \GL_n$. As a direct consequence of the above and Corollary \ref{sec:pro-unit-character} we get
\begin{corollary}
  \label{sec:exter-powers-spring-3}
  Under the assumptions of Proposition \ref{sec:exter-powers-spring}, perverse co\-ho\-mo\-lo\-gy sheaves of the cone of the map \[\operatorname{Id}_{e_c}\star\,\Phi_n^k:e_c\star\Spgr_{\wedge^{k+1}}\to e_c\star\wedge^{k+1}\Spgr_n\] are character sheaves with a central character fixed by $W$.
\end{corollary}
Note that it immediately follows that the morphisms $\operatorname{Id}_{e_c}\star\,\Phi_n^k$ are isomorphisms for $k < n - 1$: indeed, character sheaves on $\GL_n$ have full support, while it is easy to see that the source and target of $\operatorname{Id}_{e_c}\star\,\Phi_n^k$ have closed support, since a product of $k+1 < n$ matrices conjugate to a one in $U_P$ has an eigenvalue equal to 1 and lies in $\SL_n$. It now follows from \ref{item:3} that $\Phi_n^k$ is an isomorphism.

We now deduce the case $k \geq n - 1$ from the following
\begin{proposition}
  \label{sec:exter-powers-spring-1}
For any $c\in \mathcal{C}(Z(\mathrm{GL}_n))$, $e_c \star^0 \mathcal{S}_n \star^0 \mathcal{S}_{\wedge^{n-1}}$ does not contain character sheaves as direct summands.
\end{proposition}

Note that $\mathcal{S}_n \simeq \delta_e \oplus \mathcal{C}_n$, where $\mathcal{C}_n$ is the intersection cohomology complex of the minimal unipotent orbit. It follows that
\[
  \wedge^k \mathcal{S}_n \simeq \wedge^k\mathcal{C}_n \oplus \wedge^{k-1}\mathcal{C}_n.
\]

It is easy to see that the functor $e_c \star^0 -$ is monoidal and we have $\wedge^{k}(e_c \star^0\mathcal{S}_n) \simeq e_c \star^0 \wedge^{k} \mathcal{S}_n$.
From Corollary \ref{sec:exter-powers-spring-3} and the discussion immediately below it,  it follows that
\[
  e_c\star^0\wedge^{n+1}\mathcal{S}_n \simeq e_c\star^0(\wedge^{n+1}\mathcal{C}_n \oplus \wedge^{n}\mathcal{C}_n),
\]
is a character sheaf, and so is $e_c \star^0\wedge^{n}\mathcal{C}_n$. On the other hand, we have that the object $e_c\star^0\wedge^{n}\mathcal{C}_n \simeq e_c \star \wedge^{n}\mathcal{C}_n$ is a summand of $e_c\star^0\wedge^{n}\mathcal{S}_n$, which is in turn contained as a summand in
\[
    e_c\star^0\mathcal{S}_n\star^0\wedge^{n-1}\mathcal{S}_n\simeq e_c\star^0\mathcal{S}_n\star^0\mathcal{S}_{\wedge^{n-1}}.
\] 
\sloppy Thus Proposition \ref{sec:exter-powers-spring-1} together with property \ref{item:3} imply that ${\wedge^n\mathcal{C}_n \simeq 0}$, so that $\wedge^{n+1}\mathcal{S}_n \simeq 0$ and $\wedge^{n}\mathcal{S}_n \simeq \wedge^{n-1}\mathcal{C}_n$, which finishes the inductive step.
\begin{proof}[Proof of Proposition \ref{sec:exter-powers-spring-1}]
We will need the following elementary
\begin{lemma}
  \label{sec:exter-powers-spring-2}
  Let $D$ be a regular semisimple matrix in $\mathrm{SL}_n(\kk)$ with no eigenvalues equal to 1. Let $P \subset \mathrm{SL}_n$ be the parabolic subgroup fixing a line $l$. If a pair of matrices $M, N,$ with $M$ unipotent, $N \in U_P(\kk)$, and $NM = D$, exists, then $l$ contains a cyclic vector for $D$. If it exists, such a pair is unique for a given $l$. 
\end{lemma}
\begin{proof}
  Note that if $l$ does not contain a cyclic vector, then for any $N' \in U_P$, $N'D$ preserves an $(n-1)$-dimensional subspace and acts as a constant not equal to 1 on the quotient by this subspace, so can't be unipotent.

  Assume that $l$ contains a cyclic vector and that such a pair of matrices $M, N$ exists. Writing $N = \id + A$ and $M = \id + B$, with $A$ and $B$ nilpotent, note that \[n = \operatorname{rk}(D-\id)=\operatorname{rk}(AB + A + B) \leq \operatorname{rk}A+\operatorname{rk}B = 1 + \operatorname{rk}B.\] It follows that $M$ must be a regular unipotent element, such that $l$ is not in the image of $B$. Direct computation now shows that $N'M$ with $N'\in U_P(\kk)$ is unipotent if and only if $N' = \operatorname{id}$. If $M_1, N_1$ is another such pair satisfying $N_1M_1 = D,$ we have that $M_1 = (N_1^{-1}N)M$ is unipotent, so that $N_1 = N$ and $M_1 = M$, as desired.    
\end{proof}

  Since $e_c\star^0\mathcal{S}_n \star^0 \mathcal{S}_{\wedge^{n-1}}$ is the direct sum of perverse sheaves with proper closed support and
  \[
    e_c\star^0\mathcal{S}_n\star^0\cc_{\mathcal{N}}[2\dim U],
  \]
  any direct summand of the former sheaf with full support is also a direct summand of the latter.
  
  Let $\mathfrak{C}_n = \{(M,N,l): M \in \mathcal{N}, (N,l)\in\mathcal{N}_P\}$ with $m: \mathfrak{C}_n \to \GL_n$ the multiplication map $m(M,N,l) = NM$. Write
  \[
    \mathcal{S}_n\star^0\cc_{\mathcal{N}}[2\dim U] \simeq {}^p\mathcal{H}^0(m_*\cc_{\mathfrak{C}_n}[n^2+n-2]).
  \]
  Lemma \ref{sec:exter-powers-spring-2} implies that the fiber of $m$ over the generic regular semisimple element in the intersection of the fixed torus $T$ with the image of $m$ is identified with $T_{adj}= T/Z(\GL_n)$. It follows that, if $e_c\star^0\mathcal{S}_n\star^0\cc_{\mathcal{N}}[2\dim U]$ contains a character sheaf as a direct summand, it must be a character sheaf with a generic stalk of dimension $1$. We are left to consider two cases, when restriction of this summand to $\SL_n$ is a shifted constant sheaf, or a character sheaf corresponding to the Steinberg representation. Since the functors of $\ast$-convolution and (non-equivariant) hypercohomology $\mathbb{H}^0$ are both right $t$-exact, we have
  \[
    \mathbb{H}^0(\mathcal{S}_n\star^0\cc_{\mathcal{N}}[2\dim U]) \simeq  \mathbb{H}^0(\mathcal{S}_n\star \cc_{\mathcal{N}}[2\dim U]). 
  \]
  Since $\operatorname{H}^{\bullet}(\mathfrak{C}_n) \simeq \operatorname{H}^{\bullet}(\mathbb{P}^{n-1})$, we get that $\mathbb{H}^0(\mathcal{S}_n\star \cc_{\mathcal{N}}[2\dim U]) = 0$, and so the shifted constant sheaf $\cc_{\SL_n}[n^2-1]$ can not be a summand of $\mathcal{S}_n\star^0\cc_{\mathcal{N}}[2\dim U]$.

   Finally, we show that the Steinberg character sheaf can not be a direct summand of $\mathcal{S}_n\star^0\cc_{\mathcal{N}}[2\dim U]$, corresponding to $\wedge^{n+1}\mathcal{S}_n$. We recall some notations from Section \ref{sec:finite-monodr-hecke}. There we defined the completed monodromic Hecke category $\mathcal{H}(\SL_n)$, the free-monodromic tilting object $\T_{w_0}$, standard and costandard objects $\hat{\Delta}_w, \hat{\nabla}_w$ in $\mathcal{H}(\SL_n)$ from \cite{bezrukavnikovKoszulDualityKacMoody2013}. We denote the composition
   \[
     D_G^b(G) \xrightarrow{\hc} D_G^b(\yy) \xrightarrow{\operatorname{For}_T} D_G^b(G/U \times G/U) \xrightarrow{~} D^b(U\backslash G/U) \to \mathcal{H}(\SL_n)
   \]
   by $\hc'$.

   Denote by $\pi:G/U \to G/B$ the canonical projection. Let $L_w$ for $w\in W$ stand for the pullback $\pi^*\IC_w$, where $\IC_w$ is the intersection cohomology sheaf of the Bruhat $U$-orbit labeled by $w$.  
   We collect some facts about the completed category and the functor $\hc$ in the following
   \begin{proposition}
     \label{sec:exter-powers-spring-4}
     \begin{enumerate}[label=\alph*),ref=(\alph*)]
     \item\label{item:15} The functor $\hc'$ is monoidal.
     \item\label{item:16} The functor $\DD_{w_0}\star\hc'$ is $t$-exact with respect to the perverse $t$-structure, when restricted to the subcategory of unipotent character sheaves.
     \item\label{item:18} We have $L_w \star \T_{w_0} = 0$ unless $w=1$.
     \item\label{item:19} The functor $(-) \star\T_{w_0}$ of convolution with $\T_{w_0}$ is $t$-exact with respect to the perverse $t$-structure.
     \end{enumerate}
   \end{proposition}
   \begin{proof}
     \ref{item:15} follows from the monoidality of $\operatorname{For}_T$. \ref{item:16} is proved in \cite[Corollary 3.4]{bezrukavnikovCharacterDmodulesDrinfeld2012} in the setting of $D$-modules and in \cite[Corollary 7.9]{chenFormulaGeometricJacquet2017} in the $\ell$-adic setting. \ref{item:18} is \cite[Lemma 4.4.6, Lemma 4.4.11 (3)]{bezrukavnikovKoszulDualityKacMoody2013}. \ref{item:19} is a direct consequence of \ref{item:18}. 
   \end{proof}
   It is shown in \cite{hhh}, Corollary 6.2.7, that $\hc'(\mathcal{S}_n)\star \hat{\Delta}_e$ has a filtration by the objects of the form $\mathfrak{J}_w \vcentcolon= \hat{\nabla}_w\star\hat{\nabla}_{w^{-1}}$ for $w \in W^J$. It follows from Proposition \ref{braid_relations} \ref{item:14} that
   \[
     \hc'(\mathcal{S}_n)\star\T_{w_0} \simeq {\mathbf{V}} \otimes_{\Qlbar} \T_{w_0}, 
   \]
   where $\bf{V}$ is an $n$-dimensional $\Qlbar$-vector space that can be identified with $\KK^n$ using the isomorphism $\mathfrak{J}_w \star \T_{w_0} \simeq \T_{w_0}$. More generally,
   \[
     \hc'(\mathcal{S}_n^{\star k})\star\T_{w_0} \simeq \mathbf{V}^{\otimes k} \otimes_{\Qlbar} \mathcal{S}_n. 
   \]
   For any $\F \in D_G^b(G)$, by Proposition \ref{sec:exter-powers-spring-4}, we have
   \[
     \pH^0(\hc'(\F)\star\T_{w_0})\simeq \hc'(\pH^0(\F))\star \T_{w_0}.
   \]
   It follows that
   \[
     \hc'(\mathcal{S}_n^{\star^0 k})\simeq \mathbf{V}^{\otimes k} \otimes_{\Qlbar} \mathcal{S}_n,
   \]
   with $S_k$-action on the left-hand side compatible with permutation action on $\mathbf{V}^{\otimes k}$. 

   It follows that \[\hc'(\wedge^{n+1}\mathcal{S}_n)\star\T_{w_0}\simeq \wedge^{n+1}\mathbf{V}\otimes_{\Qlbar}\T_{w_0} = 0.\]  Note that for any irreducible unipotent character sheaf $L\neq L_{St}$, we have $\hc'(L)\star\T_{w_0} = 0$ by \ref{sec:exter-powers-spring-4} \ref{item:18} and  \cite[Proposition 2.7]{lusztigTruncatedConvolutionCharacter2014}. However, the category of unipotent character sheaves is not sent to 0 by the functor $\hc'(-)\star \T_{w_0}$, and so we get that $\hc'(L_{St})\star \T_{w_0} \neq 0$. This implies that the Steinberg character sheaf $L_{St}$ can not be a direct summand of $\hc'(\wedge^{n+1}\mathcal{S}_n)$. This concludes the proof. 
 \end{proof}
 \subsection{The $\RepG$-action.}
 \label{sec:repg+-repg-action}
Let $\mathcal{C}$ be a $\KK$-linear, symmetric monoidal idempotent-complete category. Then, for any object $X \in \mathcal{C}$ and positive integer $k$, we have the $S_k$-action on the object $X^k$, and, for any partition $\lambda$, Schur functors $S^{\lambda}X$ are defined as images of the corresponding projectors. The following Lemma will allow to extend this assignment to a functor $\RepG^+ \to \mathcal{C}$.

\begin{lemma}
  \label{sec:repgk+-repgk-action-1}
  Let $X \in Ob(\mathcal{C})$ be such that $\wedge^{n+1}X = 0$. Then
\begin{enumerate}[label=\alph*),ref=(\alph*)]
 \item\label{item:9} there is a monoidal functor
   \[
     \RepG^+ \to \mathcal{C}
   \]
   sending the standard $n$-dimensional representation of $\bfG$ to $X$.
 \item\label{item:10} Assume that the category $\mathcal{C}$ acts on a $\KK$-linear idempotent-complete category $\mathcal{M}$ so that $\wedge^{n}X$ acts as an invertible functor. Then the action of $\RepG^+$ on $\mathcal{M}$ given by \ref{item:9} extends to the $\RepG$-action.  
\end{enumerate}
\end{lemma}
\begin{proof}
  We will use the framework of \cite[Section 10]{deligneCategorieRepresentationsGroupe2007}. Let ${\Rep^+(\bfGL(t=n), \KK)}$ stand for the full idempotent-complete subcategory of the Deligne category $\Rep(\bfGL(t = n), \KK)$ with objects $(U, \varnothing)$ in the notations of ibid.\@ Here notation $t=n$ means that we specialize the indeterminate $t$ in the definition of $\Rep(\bfGL(t), \KK)$ to $n$. The category ${\Rep^+(\bfGL(t=n), \KK)}$ is just the idempotent completion of a category with objects $[m] \in \mathbb{N}$ and morphisms given by 1-dimensional unoriented flat bordisms.  There is a functor ${\Rep^+(\bfGL(t=n), \KK)} \to \mathcal{C}$, sending the generating object $[1]$ to $X$. The fact that $\wedge^{n+1}X = 0$ implies, by the second fundamental theorem of invariant theory of the general linear group (see, for example, \cite[Theorem 6.1]{procesiLieGroupsApproach2007}, that this functor factors through $\RepG^+$. This proves \ref{item:9}.

  We now prove \ref{item:10}. For $\mathbf{W} \in \RepG^+$, let $\phi_{\mathbf{W}}:\mathcal{M} \to \mathcal{M}$ be the corresponding functor.  Let $\mathbf{det}, \mathbf{V} \in \RepG^+$ be the determinant and standard representations, respectively. Define $\phi_{\mathbf{V}^{\vee}} =\phi_{\wedge^{n-1}\mathbf{V}} \circ \phi_{\mathbf{det}}^{-1}$. It is easy to see that this defines an action of ${\Rep(\bfGL(t=n), \KK)}, \RepG$ as in $\ref{item:9}$. 
\end{proof}
\begin{corollary}
  \label{sec:repgk+-repgk-action}
  There is an action of $\RepG$ on $\mathcal{H}_n^{\perf}$, sending the action of the standard representation to the functor of convolution with $\hc(\mathcal{S}_n)$.
\end{corollary}
\begin{proof}
  By Theorem \ref{sec:lie-group-case} and Lemma \ref{sec:repgk+-repgk-action-1} \ref{item:9}, there is a functor $\zeta:\RepG^+ \to P_G(G)$. Projecting to the category of unipotent character sheaves, we have a functor $\hat{\zeta}:\RepG^+ \to \pro P^\bfone(G)$, sending a representation $\mathbf{W}$ to the system $\left(\zeta(\mathbf{W})\star E_k^{0}\right)_{k=0}^{\infty}$ (note that the projections of $\wedge^k\mathcal{S}_n$ are perverse by \cite[Lemma 5.6.3]{hhh}). By Proposition \ref{sec:exter-powers-spring-4} \ref{item:16}, \ref{item:19}, $P^\bfone(G)$ acts on $\mathcal{H}_n^{\perf}$. Moreover, the action of $\mathbf{det} \in \RepG^+$ obtained is by $\hc(\mathcal{S}_{\wedge^n})\star\DD_e \simeq \DD_{w_0}^{\star 2}$, where the last isomorphism is \cite[Corollary 5.7.6]{hhh}. From  Proposition \ref{braid_relations} \ref{item:14} it follows that this action is by an invertible functor, and so by Lemma \ref{sec:repgk+-repgk-action-1} \ref{item:10} we get the result.
\end{proof}
Let us retain the notation from the proof of Corollary \ref{sec:repgk+-repgk-action}: for $\mathbf{W} \in \RepG^+$ let $\hat{\zeta}(\mathbf{W}) \in \pro P_G(G)$ be the corresponding unipotent character (pro-)sheaf. For an arbitrary $\mathbf{W} \in \RepG$, let $\phi_\mathbf{W} : \mathcal{H}_n^{\perf}\to\mathcal{H}_n^{\perf}$ be the corresponding functor. We have $\phi_\mathbf{W}(-) \simeq \hc(\hat{\zeta}(\mathbf{W})) \star -$ for $\mathbf{W} \in \RepG^+$.
 \section{Construction of the linear structure.}
 \label{sec:extens-funct-from-2}
\subsection{Main result and Tannakian formalism for the perfect subcategory.}
This section is occupied with the proof of the following
\begin{theorem}
 \label{sec:extens-funct-from-1} There is a monoidal functor $\varpi:\perf_\bfG(\gstb) \to \mathcal{H}_n$ satisfying the following:
  \begin{enumerate}[label=\alph*),ref=(\alph*)]
  \item $\varpi(\cO_{\gstb}) = \T_{w_0}$.
  \item $\varpi(\cO_{\gstb}(\boldsymbol\lambda, \boldsymbol\mu)) = \mathfrak{J}_{\boldsymbol\lambda}\star\T_{w_0}\star\mathfrak{J}_{\boldsymbol\mu}$ for any two weights $\boldsymbol\lambda, \boldsymbol\mu \in \RepT$.
  \item $\varpi(\mathbf{W} \otimes \cO_{\gstb}) = \phi_\mathbf{W}\T_{w_0}$ for any representation $\mathbf{W}$ of $G$.
  \end{enumerate}
\end{theorem}
We employ the general construction of \cite{bezrukavnikovTwoGeometricRealizations2016}. Let $\mathfrak{t}$ be the Lie algebra of $\bfT$.
Let $\fl = \bfG/\bfB$ and $\mathbf{Y} = \bfG/\bfU$.
Let $\mathbf{C}_{\gspr}$ be the preimage in $\g \times \mathbf{Y}$ of $\gspr$ under the natural projection ${\g \times \mathbf{Y} \to \g \times \fl}$. Since the right $\bfT$-action on $\mathbf{Y}$ is free, we have $\coh_{\bfG\times \bfT}(\mathbf{C}_{\gspr}) = \coh_{\mathbf{G}}(\gspr)$.
Let $\overline{\mathbf{Y}} = \operatorname{Spec}\Gamma(\cO_\mathbf{Y})$ be the affine
closure of $\mathbf{Y}$.

Now pick a representation $\mathbf{V_0}$ of $\bfG$, such that
\begin{enumerate}[label=\alph*),ref=(\alph*)]
\item $\mathbf{V_0}$ is a multiplicity-free sum of irreducible representations;
\item $\mathbf{Y}$ is a locally closed subvariety of $\mathbf{V_0}$, namely an orbit of a
highest-weight vector;
\item $\overline{\mathbf{Y}}$ is the closure of $\mathbf{Y}$ inside $\mathbf{V_0}$.
\end{enumerate} Choose an action of $\mathfrak{t}$ on $\mathbf{V_0}$ so that $\mathbf{t}
\in \mathfrak{t}$ acts on an irreducible summand with the highest
weight $\boldsymbol\lambda$ by the constant $\langle\boldsymbol\lambda, \mathbf{t}\rangle$.
Let $\overline{\mathbf{C}}_{\gspr}$ be the closed subscheme of $\g \times
\mathfrak{t} \times \overline{\mathbf{Y}}$ given by $$\overline{\mathbf{C}}_{\gspr} =
\{(\mathbf{x},\mathbf{t},\mathbf{v})\in \g \times \mathfrak{t} \times \overline{\mathbf{Y}} \subset \g
\times \mathfrak{t} \times \mathbf{V_0} : \mathbf{x}(\mathbf{v}) = \mathbf{t}(\mathbf{v})\}.$$

$\mathbf{C}_{\gspr}$ becomes an open subset in $\overline{\mathbf{C}}_{\gspr}$, given by
$(\mathbf{x},\mathbf{t},\mathbf{v})$ with $\mathbf{v}$ having a non-zero projection to all irreducible summands of $\mathbf{V_0}$.
Let $\overline{\mathbf{C}}_{\textrm{St}}$ be the preimage of the diagonal under
the map $\overline{\mathbf{C}}_{\gspr} \times \overline{\mathbf{C}}_{\gspr} \to \g
\times \g$. We have an open subset $\mathbf{C}_{\gstb} \subset
\overline{\mathbf{C}}_{\gstb}$ with a free $\bfT^2$-action, coming from the right $\bfT$-action on $\mathbf{Y} = \bfG/\mathbf{U}$, and an isomorphism
$\mathbf{C}_{\gstb}/\bfT^2 = \gstb.$

For a variety $\mathbf{X}$ with an action of an algebraic group $\mathbf{H}$, let
$\coh^{\mathrm{fr}}_{\mathbf{H}}(\mathbf{X})$ be the full subcategory of $\coh_{\mathbf{H}}(\mathbf{X})$, consisting of objects of the form $\mathbf{W} \otimes \cO_{\mathbf{X}}$ with $\mathbf{W} \in \Rep (\mathbf{H})$.

For a dominant weight $\boldsymbol\lambda$ let $\mathbf{V}_{\boldsymbol\lambda}\in\Rep(\bfG)$ be the irreducible representation with the highest weight $\boldsymbol\lambda$.
The following statement is \cite[Corollary 18 and Section 4.4.1]{bezrukavnikovTwoGeometricRealizations2016}.
\begin{proposition}
 \label{sec:extens-funct-from} Let $\mathcal{C}$ be a $\KK$-linear additive monoidal category. Suppose we are given
    \begin{enumerate}[label=\alph*),ref=(\alph*)]
    \item A tensor functor $F: \Rep (\bfG \times \bfT) \to \operatorname{End}(\mathcal{C})$, respecting the triangulated structure.
    \item A tensor endomorphism $E$ of $F|_{\RepG},$ satisfying
      \begin{equation*}
        E_{\mathbf{V}_1\otimes \mathbf{V}_2} = E_{\mathbf{V}_1} \otimes Id_{F(\mathbf{V}_2)} + Id_{F(\mathbf{V}_1)}\otimes E_{\mathbf{V}_2}.
      \end{equation*}
      \label{item:95}
    \item An action of $\cO_{\mathfrak{t}}$ on $F$ by endomorphisms, such that for $f \in \cO_\mathfrak{t}$ we have $f_{\mathbf{V}_1 \otimes \mathbf{V}_2} = f_{\mathbf{V}_1} \otimes Id_{F(\mathbf{V}_2)} = Id_{F(\mathbf{V}_1)} \otimes f_{\mathbf{V}_2}$.
     \label{item:96}
    \item \label{item:17} A ``highest weight arrow'' $w_{{\boldsymbol\lambda}}:F({\boldsymbol\lambda}) \to F(\mathbf{V}_{\boldsymbol\lambda})$ making the following diagrams commutative:
    \begin{equation}
      \label{eq:2}
      \begin{tikzcd} F({{\boldsymbol\lambda}})\otimes_{\mathcal{C}} F({{\boldsymbol\mu}}) \arrow[r] \arrow[d, "w_{\boldsymbol\lambda}\otimes_{\mathcal{C}}w_{\boldsymbol\mu}"'] & F\left({\boldsymbol\lambda}+{\boldsymbol\mu}\right) \arrow[d, "w_{{\boldsymbol\lambda}+{\boldsymbol\mu}}"] \\
F(\mathbf{V}_{\boldsymbol\lambda}) \otimes_{\mathcal{C}} F(\mathbf{V}_{\boldsymbol\mu}) \simeq F(\mathbf{V}_{\boldsymbol\lambda} \otimes \mathbf{V}_{\boldsymbol\mu}) \arrow[r] & F(\mathbf{V}_{{\boldsymbol\lambda} + {\boldsymbol\mu}})
      \end{tikzcd}
    \end{equation}
    \begin{equation*}
    \begin{tikzcd} F({{\boldsymbol\lambda}}) \arrow[r, "w_{{\boldsymbol\lambda}}"] \arrow[d, "{{\boldsymbol\lambda}}"'] & F(\mathbf{V}_{\boldsymbol\lambda}) \arrow[d, "E_{\mathbf{V}_{\boldsymbol\lambda}}"] \\
F({{\boldsymbol\lambda}}) \arrow[r, "w_{{\boldsymbol\lambda}}"] & F(\mathbf{V}_{\boldsymbol\lambda})
      \end{tikzcd}
      \end{equation*} where the right vertical map is the action of the element ${\boldsymbol\lambda}\in\mathfrak{t}^{\vee}\subset \cO(\mathfrak{t})$ coming from \ref{item:96}.
    \end{enumerate} Then the tensor functor $F$ extends uniquely (in the sense of \cite{bezrukavnikovTwoGeometricRealizations2016}) to an action of $\coh^{\mathrm{fr}}_{\bfG\times \bfT}(\overline{\mathbf{C}}_{\gspr})$ on $\mathcal{\mathbf{C}}$. If we have two instances of the above data, we get an action of $\coh^{\mathrm{fr}}_{\bfG^2\times \bfT^2}(\overline{\mathbf{C}}_{\gspr}\times \overline{\mathbf{C}}_{\gspr})$. In this case, assume, moreover, that
    \begin{enumerate}[resume, label=\alph*), ref=(\alph*)]
    \item the action of $\Rep (\bfG \times \bfG)$ factors through the action of $\Rep(\bfG)$ using the restriction along the diagonal embedding $\bfG \to \bfG \times \bfG$.
    \item this action of $\RepG$ is by product with central objects in $\mathcal{C}$, with endomorphism $E$ from \ref{item:95} respecting the central structure.
    \end{enumerate} Then we have an action of $\coh^{\mathrm{fr}}_{\bfG \times \bfT^2}(\overline{\mathbf{C}}_{\gstb})$, extending the above data.
  \end{proposition}

  We now indicate the structures listed in the above proposition, constructing an action of $\coh^{\mathrm{fr}}_{\bfG \times \bfT^2}(\overline{\mathbf{C}}_{\gstb})$ on $\mathcal{H}_n$.

  First, the action of $\Rep (\bfG \times \bfG)$ is given by Corollary \ref{sec:repgk+-repgk-action} applied to left and right action of $P^\bfone(G)$ on $\mathcal{H}_n^{\perf}$. The action of $\Rep (\bfT\times \bfT)$ is given by left and right convolutions with Jucys--Murphy sheaves $\mathfrak{J}_{\boldsymbol\lambda}$.

The action of $f \otimes g \in \cO_{\mathfrak{t}\times\mathfrak{t}}$ is given by the action of $f$ and $g$ coming from left and right monodromy, respectively.
  
  We construct the rest of the data and prove the needed compatibilities in the following subsections.
  
  \subsection{Highest weight arrow and monodromy endomorphism.} Let $V_n, \mathbf{V}_n$ stand for the standard $n$-dimensional representation of $\GL_n, \bfGL_n$ with fixed bases $(e_1, \dots, e_n), (\bfe_1, \dots, \bfe_n)$, respectively. For $k < n$ consider the subspace $V_k = \operatorname{span}_\kk(e_1, \dots, e_k)$, respectively $\mathbf{V}_k = \operatorname{span}_\KK(\bfe_1, \dots, \bfe_k)$, as a standard representation of $\GL_k = \GL(V_k), \bfGL_k = \bfGL(\mathbf{V}_k)$. Let $\operatorname{triv}_{\mathbf{GL}_{n-k}}$ stand for the trivial representation of $\mathbf{GL}_{n-k}$. Consider \[{\stigma}_k = \hat{\zeta}(\wedge^k\mathbf{V}_k) \boxtimes \hat{\zeta}(\operatorname{triv}_{\mathbf{GL}_{n-k}}),\] a (pro-)character sheaf on the Levi subgroup ${L}_k = \GL_k \times {\GL}_{n-k}$ of $\GL_n$. Let ${P}_k \subset \GL_n$ be the parabolic subgroup preserving the subspace $V_k$.

  Note that we abuse notation here, denoting by $\hat{\zeta}$ functors from $\operatorname{Rep}\bfGL_l^+$ with different $l$. It should always be clear what we mean.

  We have $\hat{\zeta}(\wedge^k \mathbf{V}_n) \simeq \ind_{P_k}^G(\stigma_k)$. For an arbitrary reductive group $G$, let $G^{rss}$ stand for the open subset of regular semisimple elements. $\stigma_k$ is the IC-extension of the (pro-)local system on $L_k^{rss}=\textrm{GL}_k^{rss} \times \textrm{GL}_{n-k}^{rss}$. It will be convenient to describe this (pro-)local system explicitly. Recall that the category of pro-unipotent local systems on $T$, as defined in \cite{bezrukavnikovKoszulDualityKacMoody2013}, can be identified with the category of $\KK[[\mathbf{V}_n]]$-modules. Let
  \[
    A_k = \KK[S_k \times S_{n-k}]\ltimes \KK[[\mathbf{V}_n]] = \KK[S_k \times S_{n-k}]\ltimes \KK[[\bfe_1, \dots, \bfe_n]].
  \]
  We have a functor from the category of $A_k$-modules to the category of (pro-)local systems on $L_k^{rss}$ given by restriction along the embedding $T^{rss} \to T$ and descent along the $S_k \times S_{n-k} \simeq W(L_k)$-cover $T^{rss}/T \to L_k^{rss}/L_k$.

   Note that the element $(\bfe_1+\dots +\bfe_k)\in A_k$ is central, so acts by module endomorphisms. We have
\[
  \hc(\hat{\zeta}(\wedge^k \mathbf{V}_n)) = \hc_{B}^{P_k}\hc_{P_k}^G\chi_{P_k}^G\iota_{P_k}(\stigma_k).
\]
Let $\boldsymbol\omega_k$ be the highest weight of $\wedge^k V_n$. We have
\[
  \hc_{B}^{P_k} \iota_{P_k}(\stigma_k) \simeq \mathfrak{J}_{\boldsymbol\omega_k} = \mathfrak{J}_1\star\dots \star \mathfrak{J}_k,
\]
by \cite[Lemma 6.2.4]{hhh}. We also have a natural unit morphism
\[
  r_k: \iota_{P_k}(\stigma_k) \to \hc_{P_k}^G\chi_{P_k}^G\iota_{P_k}(\stigma_k).
\]

Combining, we get an endomorphism
\[
  \ind_{P_k}^{G}(\bfe_1 + \dots + \bfe_k) =: \varepsilon_k \in \operatorname{End}(\hat{\zeta}(\wedge^k \mathbf{V}_n)),
\]
which makes the following diagram commutative:

\[
  \begin{tikzcd}
  \mathfrak{J}_{{\boldsymbol\omega}_k} \arrow[r, "w_{{\boldsymbol\omega}_k}"] \arrow [d, "{\boldsymbol\omega}_{k}"'] &\hat{\zeta}(\wedge^k \mathbf{V}_n)\arrow[d, "E_k"] \\
    \mathfrak{J}_{{\boldsymbol\omega}_k} \arrow[r, "w_{{\boldsymbol\omega}_k}"] & \hat{\zeta}(\wedge^k \mathbf{V}_n) 
  \end{tikzcd}
\]

      Here $w_{{\boldsymbol\omega}_k} = \hc_B^{P_k}(r_k), E_k = \hc(\varepsilon_k)$, and ${\boldsymbol\omega}_k$ comes from the monodromy action on $\mathfrak{J}_{{\boldsymbol\omega}_k}$.
      
      We define $w_{{\boldsymbol\omega}_k}$ to be the highest weight arrows for the images of fundamental representations $\wedge^k\mathbf{V}_n$. For a dominant weight $\boldsymbol\lambda = \sum a_i{\boldsymbol\omega}_i,$ we must have, from the condition \ref{item:17} of Proposition \ref{sec:extens-funct-from}, $w_\lambda = \pi_{\boldsymbol\lambda} ( w_{{\boldsymbol\omega}_1}^{\star a_1} \star \dots \star w_{\boldsymbol\omega_n}^{\star a_n})$, where $\pi_{\boldsymbol\lambda}$ stands for the projection $\pi_{\boldsymbol\lambda}:\hat{\zeta}(\otimes \mathbf{V}_{\boldsymbol{\omega}_k}^{\otimes a_i}) \to \hat{\zeta}(\mathbf{V}_{\boldsymbol\lambda})$.

      Now the condition \ref{item:95} of Proposition
\ref{sec:extens-funct-from} recovers $E_\mathbf{V}$ from $E_1$ in a unique way.
The only thing left to check is
      \begin{lemma} The morphisms $E_{{\boldsymbol\omega}_k}$ obtained from $E_1$ by imposing
condition $\ref{item:95}$ are equal to the morphisms $E_k$ defined above.
      \end{lemma}
      \begin{proof} Recall that we have a natural morphism of functors
        \[
          \ind^G_{P_1}\res^G_{P_1}(-) \to \Spgr_n\star(-),
        \]
        which gives a morphism
        \[
 \left(\ind^G_{P_1}\res^G_{P_1}\right)^{\circ k}(-)\to \Spgr_n^{\star k}\star(-).
        \]
        By the proof of Proposition \ref{sec:exter-powers-spring}, it restricts to an isomorphism on the $S_k$-isotypic component corresponding to the sign representation. On the right-hand side this isotypic component is identified with the convolution $\hat{\zeta}(\wedge^k\mathbf{V}_n) \star (-)$. On the left-hand side, the $S_k$-action can be described explicitly as follows.  If a character sheaf $\F$ is given by the IC{-extension} of the local system corresponding to the $A_n$-module $M$, then $\ind^G_P\res^G_P\F$ is given by the IC-extension of the local system corresponding to a $A_n$-module $\KK[S_n] \otimes_{\KK[S_{n-1}]} M$. Choosing an isomorphism of $S_n$-modules
        \[
          \KK[S_n]^{\otimes_{\KK[S_{n-1}]}k}\otimes_{\KK[S_{n-1}]}\operatorname{triv}_{S_n} \simeq \mathbf{V}_n^{\otimes k}
        \]
        and comparing the actions under this identification, we get
that the action of $E_{{\boldsymbol\omega}_k}$ on the corresponding isotypic component is the same as given by $E_{k}$.
      \end{proof}
      
\subsection{Filtration by Jucys--Murphy sheaves.}
\label{sec:filtr-jucys-murphy} We have that $P_k \supset B$ is the parabolic subgroup corresponding to the subset \[J = \{s_1, \dots, s_{k-1}, s_{k+1}, \dots, s_{n-1}\} \subset I.\] By \cite[Theorem 6.2.1]{hhh}, we get
\begin{lemma}
  $\hc(\hat{\zeta}(\wedge^k\mathbf{V}_n)) \in \langle\mathfrak{J}_{{\boldsymbol\lambda}'}\rangle_{{\boldsymbol\lambda}' \in W{\boldsymbol\omega}_k}$.
\end{lemma}

\subsection{Passage from $\overline{\mathbf{C}}_{\mathbf{St}_{\mathfrak{g}}}$ to ${\mathbf{St}_{\mathfrak{g}}}.$}
We have constructed an action of $\coh^{\mathrm{fr}}_{\bfG\times \bfT^2}(\overline{\mathbf{C}}_{\gstb})$ on $\mathcal{H}_n$. Consider the action on the object $\hat{\mathcal{T}}_{w_0}$. This gives a functor
\[\varpi_{\textrm{fr}}:\coh^{\mathrm{fr}}_{\bfG\times \bfT^2}(\overline{\mathbf{C}}_{\gstb})\to \mathcal{H}_n.\]

Since $\Rep (\bfG \times \bfT^2)$ acts by convolution with objects having a filtration by products of costandard objects $\NN_w$, Proposition \ref{braid_relations} gives that ${\varpi}_{\textrm{fr}}$ in fact sends $\coh^{\mathrm{fr}}_{\bfG\times \bfT^2}(\overline{\mathbf{C}}_{\gstb})$ to $\Tilt^{\wedge}(G)$. Thus, we get a functor
\[
  \varpi_{\textrm{fr}}:\Ho^b(\coh^{\mathrm{fr}}_{\bfG\times \bfT^2}(\overline{\mathbf{C}}_{\gstb})) \to \Ho^b({\Tilt^{\wedge}(G)}) \simeq \mathcal{H}_n.
\] Finally, we employ the following proposition from \cite{zbMATH05601711}. In any tensor category over a field of characteristic zero one can construct a Koszul complex $K_{\phi,d}$ associated with any morphism $\phi\colon L \to V$ and $d \in
\mathbb{Z}^{>0}$, namely the complex
\(
  0 \to Sym^d(L) \to \dots  {\to \wedge^{d-i}(V)\otimes Sym^i(L) \to} \dots \to \wedge^{d-1}(V)\otimes L \to \wedge^d(V) \to 0,
\)
with differential induced by $\phi$.

Let $K_{{\boldsymbol\lambda}}$ be the Koszul complex $K_{w_{{\boldsymbol\lambda}},\dim \bf{V}_{{\boldsymbol\lambda}}}$ in $\coh^{\mathrm{fr}}_{\bfG\times \bfT^2}(\overline{\mathbf{C}}_{\gstb})$. We will need the following proposition, \cite[Lemma 20]{zbMATH05601711}.
\begin{proposition}
  \label{sec:pass-from-overl}
  Assume there is an action 
  \[
    a\colon\Ho^b(\coh^{\mathrm{fr}}_{\bfG\times \bfT^2}(\overline{\mathbf{C}}_{\gstb}))\to\mathrm{End}(\mathcal{C})
  \]
 of $\Ho^b(\coh^{\mathrm{fr}}_{\bfG\times \bfT^2}(\overline{\mathbf{C}}_{\gstb}))$ on a $\KK$-linear idempotent-complete triangulated category $\mathcal{C}$, such that $a(K_{{\boldsymbol\lambda}})x = 0$ for all ${\boldsymbol\lambda}$ and some object $x \in \operatorname{Ob}(\mathcal{C})$. Then the functor
  \[
    F_{\textrm{fr}}:\Ho^b(\coh^{\mathrm{fr}}_{\bfG \times \bfT^2}\overline{\mathbf{C}}_{\gstb})\to\mathcal{C}, y \mapsto a(y)x
  \]
  factors through the functor
  \[
  F: \perf_{\bfG}({\gstb}) \to \mathcal{C}.
  \]
\end{proposition}
The functor $\varpi_{\textrm{fr}}$ satisfies the condition of the above proposition, because, on an object of the form $\mathbf{W} \otimes \cO_{\overline{C}_{\gstb}}$ with $\mathbf{W} \in \Rep (\bfG \times \bfT^2)$, it is given by $\mathbf{W} \otimes \T_{w_0}$. We have
\[
  \varpi_{\textrm{fr}}(K_{{\boldsymbol\lambda}}) = k_{{\boldsymbol\lambda}}\otimes {\mathcal{T}}_{w_0} = 0,
\]
where $k_{{\boldsymbol\lambda}}$ is the (obviously acyclic) Koszul complex in $\Rep (\bfG\times \bfT^2)$. Applying Proposition \ref{sec:pass-from-overl} we get a functor
\[
  \varpi: \perf_{\bfG}({\gstb}) \to \mathcal{H}_n.
\]
\subsection{Compatibility with the monoidal structure.}
\label{sec:comp-with-mono}
We can identify the category $\mathcal{H}_n^{\perf}$ with the bounded homotopy category of free $\cO_{\hat{0}}(\mathfrak{t}\times_{\mathfrak{t}/W}{\mathfrak{t}})$-modules, where the subscript in $\cO_{\hat{0}}$ stands for the formal completion at $(0, 0) \in \mathfrak{t}\times_{\mathfrak{t}/W}{\mathfrak{t}}$ \cite[Theorem 9.1]{zbMATH07610537}. From the computation above it is easy to see that, passing through this equivalence, $\varpi$ can be identified with the functor of restriction to the Steinberg cross-section $\mathfrak{t} \times _{\mathfrak{t}/W}\mathfrak{t} \to \gstb$. Since the latter is monoidal, and the image of $\perf_{\bfG}(\gstb)$ is in $\mathcal{H}_n^{\perf}$, we get the result.
\printbibliography
\Addresses
\end{document}